\documentclass[10pt]{amsart}
\usepackage{amsmath}
\usepackage{amssymb} 
\usepackage{amsthm}
\usepackage{graphicx}

 \newtheorem{thm}{Theorem}
 \newtheorem{cor}{Corollary}

 \newtheorem{prop}{Proposition}
 \newtheorem{lem}{Lemma}
 {\theoremstyle{definition}
 \newtheorem{conj*}{Conjecture}}
 {\theoremstyle{definition}
 }
 {\theoremstyle{definition}
 \newtheorem{rem}{Remark}}
 {\theoremstyle{definition}
  \newtheorem{defn}{Definition}}
 {\theoremstyle{definition}
 \newtheorem{exam}{Example}}

\begin{document}  
\newcommand{\h}{ {\mathbb{H}}^{2} }
 \newcommand{\R}{ \mathbb{R}}
  \newcommand{\bk}{\mathbf{k}}
 \newcommand{\C}{ \mathcal{C}}
 \newcommand{\lo}{ <_{\textsf{left}}}
 \newcommand{\ro}{ <_{\textsf{right}}}
  \newcommand{\Code}{ \textsf{Code}}


\title{On finite Thurston-type orderings of braid groups}
\author{Tetsuya Ito}
\address{Graduate School of Mathematical Science, University of Tokyo, Japan}
\email{tetitoh@ms.u-tokyo.ac.jp}

\begin{abstract} 
 We prove that for any finite Thurston-type ordering $<_{T}$ on the braid group $B_{n}$, the restriction to the positive braid monoid $(B_{n}^{+},<_{T})$ is a well-ordered set of order type $\omega^{\omega^{n-2}}$. The proof uses a combinatorial description of the ordering $<_{T}$. Our combinatorial description is based on a new normal form for positive braids which we call the $\C$-normal form. It can be seen as a generalization of Burckel's normal form and Dehornoy's $\Phi$-normal form (alternating normal form). 
\end{abstract}
 \maketitle

\section{Introduction}

 The braid group $B_{n}$ is a group defined by the presentation
\[
B_{n} = 
\left\langle
\sigma_{1},\sigma_{2},\ldots ,\sigma_{n-1}
\left|
\begin{array}{ll}
\sigma_{i}\sigma_{j}=\sigma_{j}\sigma_{i} & |i-j|\geq 2 \\
\sigma_{i}\sigma_{j}\sigma_{i}=\sigma_{j}\sigma_{i}\sigma_{j} & |i-j|=1 \\
\end{array}
\right.
\right\rangle
\]

  The monoid generated by the positive generators $\{\sigma_{1},\sigma_{2},\ldots,\sigma_{n-1}\}$ is called {\it the positive braid monoid} and is denoted by $B_{n}^{+}$. An element of $B_{n}^{+}$ is called a {\it positive braid}.
   The braid group $B_{n}$, first introduced by Artin, appears in various fields of mathematics and have been studied from many points of view.

In 1990's a new feature of the braid groups, the left-invariant total orderings called {\it the Dehornoy ordering} was discovered by Dehornoy \cite{d1}. After the discovery of the orderings, numerous studies have been done from both algebraic and geometrical prospectives and the subject of braid orderings and related topics have been developed rapidly \cite{ddrw1},\cite{ddrw2}.

   In \cite{sw}, infinite families of left-invariant total orderings of the braid groups, called {\it Thurston type orderings} are constructed via the hyperbolic geometry. This family of orderings contains the Dehornoy ordering, hence it is an extension of the Dehornoy ordering. It is known that Thurston type orderings have a property called the Property $S$, and consequently they define the well-ordering when they are restricted to the positive braid monoid $B_{n}^{+}$. 

   One of the problems of braid orderings is to determine the order type of the well-ordered set $(B_{n}^{+},<)$ for a Thurston type ordering $<$, and to construct a method to compute ordinals for each element in $B_{n}^{+}$. For the Dehornoy ordering case, Burckel solved these problems by using a correspondence between positive braid words and rooted labeled trees \cite{b1},\cite{b2}. He introduced a normal form of a positive braid called {\it Burckel's normal form}, which makes it possible to compare the Dehornoy ordering and to compute ordinals. In \cite{d2}, Dehornoy gives alternative description of Burckel's normal form by introducing the $\Phi$-normal forms (alternate normal form) of positive braids. He utilized the alternate decomposition, which uses the Garside structure of braid groups. He showed that for the positive braid monoid, the $\Phi$-normal form coincides with Burckel's normal form. However, the connection between the $\Phi$-normal forms and the Dehornoy ordering is indirect. It depends on Burckel's result, so there still lie gaps to understand the Dehornoy ordering on positive braid monoid directly.

  In this paper, we give a new combinatorial description for special kinds of Thurston type orderings, called {\it finite type orderings}. The Dehornoy ordering is a typical example of finite type orderings so our results extend the results of Dehornoy and Burckel. 
  
For a positive braid word $W$ with additional structure, which we call a {\it tower of subword decompositions}, we define the code, which is a sequence of parenthesized non negative integers.
The set of codes has two lexicographical orderings $\lo$ and $\ro$. The $\C$-normal form of a positive braid $\beta$ is defined as a positive word representative of $\beta$ whose code is maximal with respect to the ordering $\ro$. This maximal code is denoted by $\C(\beta;<)$.

 Using these notions, our main theorem is stated as follows. 

\begin{thm}
\label{thm:main}
Let $<$ be a finite Thurston type ordering of $B_{n}$.
Then for $\alpha,\beta \in B_{n}^{+}$, $\alpha < \beta$ holds if and only if $\C(\alpha;<) \lo \C(\beta;<)$ holds.
\end{thm} 

Thus on the positive braid monoid $B_{n}^{+}$, the structure of finite Thurston type orderings is simple. It can be regarded as a composition of two lexicographical orderings.

This $\C$-normal form description of finite Thurston type orderings makes it possible to compare the ordering of braids by purely algebraic and combinatorial ways.
Using the $\C$-normal form description, we determine the order type of the well-ordered set $(B_{n}^{+},<)$ for all finite Thurston type orderings $<$. This result is a partial answer to the question listed in \cite{ddrw1} and \cite{ddrw2} and provides an alternative proof of Burckel's results.

\begin{thm}
\label{thm:ordertype}
Let $<$ be a finite Thurston type ordering of $B_{n}$.
Then the order type of the well-ordered set $(B_{n}^{+},<)$ is $\omega^{\omega^{n-2}}$.
\end{thm}

The $\C$-normal form can be seen as an extension of Burckel's normal form and Dehornoy's $\Phi$-normal form. 
We define the tail-twisted $\Phi$-normal form of positive braids, which is an extension of the $\Phi$-normal form. We show the tail-twisted $\Phi$-normal forms coincide with the $\C$-normal forms for special kinds of finite type orderings, called {\it normal} finite type orderings. Using this algebraic formulation of the $\C$-normal forms, we give a computational complexity of $\C$-normal forms and finite Thurston type orderings.

\begin{thm}
\label{thm:complexity}
   Let $<$ be an arbitrary finite Thurston type ordering of $B_{n}$.
\begin{enumerate}
\item The $\C$-Normal form $W(\beta;<)$ of a positive $n$-braid $\beta$ is computed by the time $O(l^{2}n\log n)$, where $l$ is the word length of $\beta$.
\item For an $n$-braid (not necessarily positive) $\beta$, whether $\beta>1$ holds or not can be decided by the time $O(l^{2}n^{3}\log n)$ where $l$ is the word length of $\beta$. 
\end{enumerate}
\end{thm}

   Thus, the $\C$-normal form argument provides an efficient algorithm to compare finite Thurston type orderings in terms of Artin generators.
  
The plan of the paper is as follows.
In section 2, we summarize the definitions and basic facts about Thurston-type orderings of the braid groups. In section 3 we introduce the code of braid words and define the $\C$-normal form. The proof of Theorem \ref{thm:main} and Theorem \ref{thm:ordertype} will be given in section 4. Section 5 is devoted to the study of the relationships between the $\C$-normal form and Dehornoy's $\Phi$-normal form. We define the tail-twisted $\Phi$-normal form and show that it coincides with the $\C$-normal form, and prove Theorem \ref{thm:complexity}.

Finally we mention other aspects of our code constructions. There is another submonoid $B_{n}^{+*}$ of $B_{n}$ called {\it the dual braid monoid}, which has similar properties of the positive braid monoid $B_{n}^{+}$. Indeed, the dual braid monoid also defines the Garside structure on the braid group $B_{n}$, and two monoids $B_{n}^{+}$ and $B_{n}^{+*}$ are closely related \cite{be},\cite{bkl}. Fromentin studied the Dehornoy orderings via the dual braid monoids using the rotating normal forms, which corresponds to Dehornoy's $\Phi$-normal form in dual braid monoids settings \cite{f1},\cite{f2},\cite{f3}. Our code constructions and the $\C$-normal forms equally work for the dual braid monoids with an appropriate modification. In particular, all results in this paper holds for dual braid monoids as well. However, in the dual braid monoid case we need a new method to prove the counterpart of Theorem \ref{thm:main}, the $\C$-normal form description of the ordering. We will describe the $\C$-normal form construction for the dual braid monoid in the subsequent paper \cite{i2}.\\ 

\textbf{Acknowledgments.}
   The author would like to express his gratitude to Professor Toshitake Kohno for his encouragement and helpful suggestions. He also wishes to thank to Bert Wiest and Matthieu Calvez for careful reading of an early version of the paper, and for pointing out mistakes. This research was supported by JSPS Research Fellowships for Young Scientists.
\setcounter{thm}{2}

\section{Thurston type ordering}

\subsection{Construction via hyperbolic geometry}

 First we briefly review a construction of Thurston type orderings. For details, see \cite{sw}.
  
 Let $D_{n}$ be the $n$-punctured disc. The braid group $B_{n}$ is naturally identified with the relative mapping class group $MCG(D_{n},\partial D_{n})$, which is the group of isotopy classes of homeomorphisms of $D_{n}$ whose restrictions on $\partial D_{n}$ are identity maps.
 The generator $\sigma_{i}$ corresponds to the isotopy class of the positive half Dehn-twist along the straight line connecting the $i$-th and the $(i+1)$-st puncture points.
 
 Let us choose a complete hyperbolic metric on $D_{n}$.
Then the universal covering $\pi:\widetilde{D_{n}} \rightarrow D_{n}$ is isometrically embedded in the  hyperbolic plane $\h$. By attaching the points at infinity to $\widetilde{D_{n}}$, we obtain a topological disc. By abuse of notation, we still denote this disc by the same notation $\widetilde{D_{n}}$.
Take a lift of the base point $\widetilde{*}$ and let $C$ be the connected component of $\pi^{-1}(\partial D_{n})$ which contains $\widetilde{*}$. Then, $\partial \widetilde{D_{n}} \backslash C$ is identified with the real line $\R$.

   For a point $x \in \partial \widetilde{D_{n}} \backslash C$, we define the relation $<_{x}$ on the braid group $B_{n}$ as follows. For $[f],[g] \in B_{n}$, let us choose their representative homeomorphisms $f,g:D_{n} \rightarrow D_{n}$ and take their lifts $\widetilde{f},\widetilde{g}: \widetilde{D_{n}}\rightarrow \widetilde{D_{n}}$ so that they fix the point $\widetilde{*}$.
We define $[f] <_{x} [g]$ if $\widetilde{f}(x) < \widetilde{g}(x)$ holds under an orientation-preserving identification of $\partial \widetilde{D_{n}} \backslash C$ with the real line $\R$. Although there are many choices of representative homeomorphisms $f$ and $g$, the restrictions of their lifts are independent of these choices, so $<_{x}$ is well-defined. The relation $<_{x}$ defines a left-invariant partial ordering of $B_{n}$. If we choose a point $x$ in a nice way, $<_{x}$ becomes a total ordering.  
We call left-invariant total orderings constructed in this way {\it Thurston type orderings}.  Although there are many choices of a hyperbolic metric and a base point in the above construction, constructed families of orderings are independent of these choices. 

   It is known that Thurston type orderings fall into two types, {\it finite type} and {\it infinite type}. In next section, we formulate a finite type ordering as an ordering which can be described by curve diagrams.

\subsection{Construction of finite Thurston type ordering via curve diagram}

  In this section we give an alternative description of a special kind of Thurston type orderings using a curve diagram. The curve diagram formulation gives more accessible meanings and is closely related to our combinatorial description. 

Let $\Gamma$ be a diagram in $D_{n}$, which consists of a union of closed oriented, embedded arcs $\Gamma_{1},\Gamma_{2},\ldots,\Gamma_{n-1}$ satisfying the following properties.
\begin{itemize}
\item Interiors of $\Gamma_{i}$ are disjoint to each other.
\item The initial point of $\Gamma_{i}$ lies at $(\bigcup _{j=1}^{i-1}\Gamma_{j}) \cup (\partial D_{n})$ and the end point of $\Gamma_{i}$ lies at puncture points or $(\bigcup _{j=1}^{i-1}\Gamma_{j}) \cup (\partial D_{n})\cup (\textrm{Interior of } \Gamma_{i} )$.
\item Each component of $ D_{n}\backslash \Gamma$ is a disc or an one-punctured disc.
\end{itemize}

   We call such a diagram {\it curve diagram}. See Figure \ref{fig:curvediagram}.
We say two curve diagrams $\Gamma$ and $\Gamma'$ are tight if they form no bigon. A bigon is an embedded disc whose boundary consists of two subarcs $\gamma \subset \Gamma_{i}$, $\gamma' \subset \Gamma'_{j}$. If we equip a complete hyperbolic metric on $D_{n}$ and realize both $\Gamma$ and $\Gamma'$ as a union of geodesics, then $\Gamma$ and $\Gamma'$ become tight. Thus, we can always isotope given two curve diagrams so that they are tight.

 For distinct braids $\alpha,\beta \in B_{n}$, if we put the image of the curve diagrams $\alpha(\Gamma)$ and $\beta(\Gamma)$ tight, then they must diverge at some point. We define $\alpha <_{\Gamma} \beta$ if the image $\beta(\Gamma)$ moves the left side of $\alpha(\Gamma)$ at the first divergence point. The relation $<_{\Gamma}$ defines a left-invariant total ordering of $B_{n}$. We say a Thurston-type ordering $<$ is of {\it finite type} if it coincides with the ordering $<_{\Gamma}$ for some curve diagram $\Gamma$. Although this definition of finite type ordering is different from that in \cite{ddrw1},\cite{ddrw2}, \cite{sw}, this definition agrees with the usual definition. Conversely, every ordering constructed by a curve diagram is always a Thurston-type ordering.	
 
   We remark that the notion of curve diagrams and its defining orderings are generalized to other mapping class groups for arbitrary surface with non-empty boundaries. See \cite{rw} for such a construction of left-invariant total orderings on mapping class groups. 

We say a curve diagram $\Gamma$ is {\it normal} if all arcs $\{ \Gamma_{i} \}$ are vertical and oriented upwards. Here we say an arc $\Gamma_{i}$ is vertical if the arc $\Gamma_{i}$ is written by the equation $x= \textrm{constant}$ by using the $x$-$y$ coordinate, regarding $D_{n}$ as the unit disc removed puncture points lying on the $x$-axis. See Figure \ref{fig:curvediagram}.

 For a normal curve diagram $\Gamma$ of $D_{n}$, we define the integer $k(i)$ so that $\Gamma_{i}$ lies between the $k(i)$-th and the $(k(i)+1)$-st punctures. Then the normal curve diagram $\Gamma$ is represented as the permutation of $n-1$ integers $\bk=\{k(1),\ldots,k(n-1)\}$. A finite Thurston type ordering $<$ is called {\it normal} if $< = <_{\Gamma}$ for some normal curve diagram $\Gamma$. 

 In this paper, we mainly consider a normal finite Thurston-type ordering. This is justified by the following facts.
 
 Two left invariant total orderings $<$ and $<'$ of a group $G$ are called {\it conjugate} if there exists an element $g \in G$ such that $f<h$ is equivalent to $fg<'hg$ for all $f,h \in G$. We call such $g$ a {\it conjugating element} between $<$ and $<'$. In \cite{sw}, it is shown that the number of conjugacy classes of finite Thurston type orderings on $B_{n}$ is finite and each conjugacy class of orderings contains at least one normal finite type ordering. 
 
 For the braid group $B_{n}$ and the positive braid monoid $B_{n}^{+}$, there is a good property which makes it easier to handle conjugate orderings.
 
\begin{lem}
 \label{lem:conjpos}
 Let $<$ and $<'$ be two left-invariant total orderings of the braid group $B_{n}$.
 If $<$ and $<'$ are conjugate, then we can choose a conjugating element $\alpha$ as a positive braid.
 \end{lem}
 
 \begin{proof}
  Let $\alpha'$ be a conjugating element between $<$ and $<'$, and let $\Delta$ be the Garside fundamental braid, defined by 
\[\Delta= (\sigma_{1}\sigma_{2}\cdots\sigma_{n-1})(\sigma_{1}\sigma_{2}\cdots\sigma_{n-2})\cdots(\sigma_{1}\sigma_{2})(\sigma_{1}).\]
It is known that $\Delta^{2}$ is central in $B_{n}$ and $\Delta^{2p}\alpha'$ is a positive braid for sufficiently large $p>0$ \cite{bi}. Thus $\Delta^{2p}\alpha'$ is a positive conjugating element if $p$ is sufficiently large.  
 \end{proof}

 We closed this section by giving examples of normal finite Thurston type orderings.

\begin{exam}[The Dehornoy ordering and the reverse of the Dehornoy ordering]
\label{exam:Dtype}

   Let $\Gamma_{D}$ be a curve diagram as shown in the left diagram of Figure \ref{fig:curvediagram}, which corresponds to the trivial permutation $\{1,2,\ldots,n-1\}$. The ordering $<_{D}$ defined by $\Gamma_{D}$ is called the {\it Dehornoy ordering}.
   
Algebraically, the Dehornoy ordering $<_{D}$ is defined as follows.
For $\alpha,\beta \in B_{n}$, we define $\alpha <_{D} \beta$ if $\alpha^{-1}\beta$ admits a word representative which contains no $\sigma_{1}^{\pm1},\sigma_{2}^{\pm1},\cdots \sigma_{i-1}^{\pm1},\sigma_{i}^{-1}$ but contains at least one $\sigma_{i}$ for some $i$ \cite{d1}.

 Next we consider the curve diagram $\Gamma_{D'}$ in the middle diagram of Figure \ref{fig:curvediagram}, which corresponds to the permutation $\{n-1,n-2,\ldots,1\}$. 
 The ordering $<_{D'}$ defined by $\Gamma_{D'}$ is called the {\it reverse of the Dehornoy ordering}.
 
In some papers (for example, \cite{d2}) the Dehornoy ordering is used to represent the reverse of the Dehornoy ordering. This is not a serious difference because these two orderings are conjugate by the Garside fundamental braid $\Delta$. That is, $\alpha <_{D'} \beta$ is equivalent to $\alpha\Delta <_{D} \beta\Delta$. 

\end{exam}

\begin{exam}
\label{exam:Ttype}
Let us present the most simple finite Thurston type ordering which is not conjugate to the Dehornoy ordering.
Let $\Gamma$ be the curve diagram of $D_{4}$, as shown in the right diagram of Figure \ref{fig:curvediagram}, corresponding to the permutation $(2,1,3)$. It is known that the ordering defined by $\Gamma$ is not conjugate to the Dehornoy ordering $<_{D}$ \cite{ddrw1},\cite{ddrw2}.
\end{exam}
\begin{figure}[htbp]
 \begin{center}
\includegraphics[width=100mm]{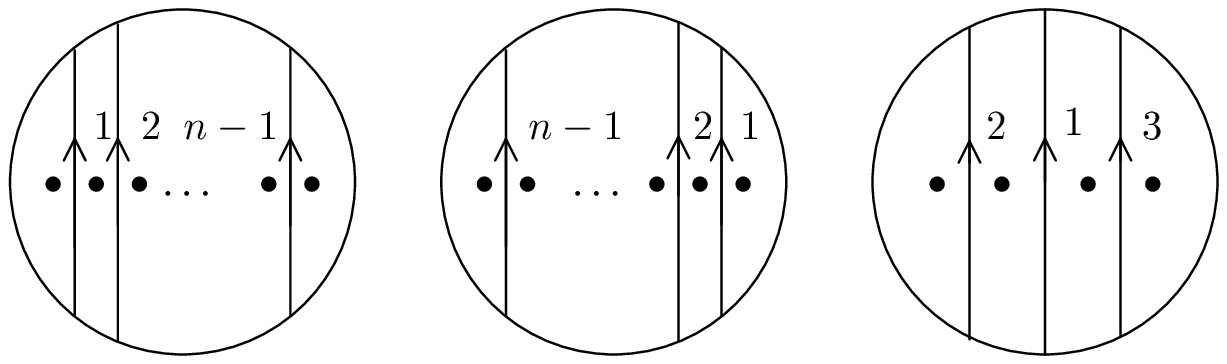}
 \end{center}
 \caption{Examples of curve diagrams}
 \label{fig:curvediagram}
\end{figure}

\subsection{Property $S$ of Thurston type orderings}

   Thurston type orderings are generalization of the Dehornoy ordering. Many properties of the Dehornoy orderings hold for Thurston type orderings as well. The most important property we use in the paper is the property $S$.
  
\begin{prop}[Property $S$ \cite{ddrw1},\cite{ddrw2}]
Let $<$ be a Thurston type ordering of $B_{n}$. Then
\[ \alpha\sigma_{i}\beta > \alpha\beta > \alpha\sigma_{i}^{-1}\beta \] 
holds for all $\alpha,\beta \in B_{n}$ and $1 \leq i \leq n-1$.
\end{prop} 

Roughly speaking, the property $S$ means we can freely insert positive generators $\{\sigma_{i}\}$ at arbitrary points without decreasing the ordering. The property $S$ will be essentially used to prove our description of orderings. 

 Using property $S$, we can prove the following corollary by the same argument as for the Dehornoy ordering case.
(see \cite[Chapter1]{ddrw1})
\begin{cor}
For every Thurston type ordering $<$ of the braid group $B_{n}$, $(B_{n}^{+},<)$ is a well-ordered set. 
\end{cor}

\section{Code and the $\C$-normal form}

    In this section, we define the codes and the $\C$-normal forms and provide various examples. First we consider a normal finite Thurston-type ordering and after we extend the results about normal orderings to general one.

 Let $<$ be a normal finite Thurston type ordering on $B_{n}$, define by the curve diagram $\Gamma$ which corresponding to the permutation $\bk=(k(1),k(2),\cdots,k(n-1))$. 
 Throughout this section we use the following notation.
 WE always denote $k(1)$ by $k$.
 Let $D:B_{n}^{+}\rightarrow B_{n}^{+}$ be a flip homomorphism defined by the conjugate of the Garside fundamental braid $\Delta$. That is, the map $D$ is defined by $D(\alpha)= \Delta^{-1} \alpha \Delta$ or equivalently, $D(\sigma_{i}) = \sigma_{n-i}$.
 Similarly, we define the shift map $Sh:B_{n} \rightarrow B_{n-1}$ by $Sh(\sigma_{i}) = \sigma_{i-1}$ for $i>1$ and $Sh(\sigma_{1})=1$.
 For positive braids $ \alpha_{1},\alpha_{2},\ldots,\alpha_{j} \in B_{n}^{+}$, we denote the submonoid of $B_{n}^{+}$ generated by $\{\alpha_{1},\alpha_{2},\ldots,\alpha_{j}\}$ by $\langle \alpha_{1},\alpha_{2},\ldots,\alpha_{j} \rangle$.

\subsection{Definitions of code and the $\C$-normal form}

First we define the code of a positive braid word with additional information which we call the {\it tower of subword decomposition}.

Our starting point is the positive $2$-braid monoid $B_{2}^{+}$.
There is only one normal finite Thurston type ordering, namely the Dehornoy ordering $<_{D}$ of $B_{2}^{+}$.
For a positive 2-braid word $W=\sigma_{1}^{a}$, we define the code of $W$ with respect to the ordering $<_{D}$ by the positive integer $(a)$. That is, we define $\C(\sigma_{1}^{a}; <_{D}) = (a)$.

 Assume that we have already defined the code of positive $i$-braid words with respect to arbitrary normal finite Thurston type orderings of $B_{i}$ for all $i<n$.
 Let $<$ be a normal finite Thurston type ordering of $B_{n}$.
  Now we define the code of positive $n$-braid word $W$ as follows.
   
 First of all, for $1 \leq j \leq n-2$, let us define two integers $m_{j}$ and $M_{j}$ by 
\[
\left\{
\begin{array}{l}
m_{j} = \max \{ 1,k(i)+1 \:|\: i=1,2,\cdots,j, \; k(i)<k(j+1)\} \\
M_{j}= \min \{ n,k(i) \:|\: i=1,2,\cdots,j,\; k(i)>k(j+1)\}.
\end{array}
\right.
\]
 and let $I_{j}=\{m_{j},m_{j}+1,\ldots,M_{j}\}$.
The set $I_{j}$ is nothing but the set of integers of the connected component of $[1,n] - \{k(1)+\frac{1}{2}, \ldots, k(i)+\frac{1}{2}\}$ which contains $k(i+1)$. 
For example, $m_{n-2}=k(n-1)$, $M_{n-2}=k(n-1)+1$.

Now take a decomposition of the word $W$ into the products of subwords as 
\[ W = A_{m}A_{m-1}\cdots A_{0}A_{-1} \]
 where
\[
\left\{
\begin{array}{l}
A_{-1} \in \langle \sigma_{1},\sigma_{2},\ldots, \sigma_{k-1},\sigma_{k+1},\ldots,\sigma_{n-1} \rangle \\

 A_{i} \in \langle \sigma_{1},\sigma_{2},\ldots,\sigma_{n-2} \rangle \;\;\; i \geq 0: odd \\

 A_{i} \in \langle \sigma_{2},\sigma_{3},\ldots,\sigma_{n-1} \rangle \;\;\; i \geq 0: even .
\end{array}
\right. .
\]
and take a subword decomposition of $A_{-1}$ as 
\[ A_{-1} = X_{1}X_{2}\cdots X_{n-1} \]
where $X_{j} \in \langle \sigma_{m_{j}},\sigma_{m_{j}+1},\ldots,\sigma_{M_{j}-1} \rangle$.

Let $<_{D}$ be the Dehornoy ordering of $B_{n-1}$.
Observe that $D^{i-1}(A_{i}) \in B_{n-1}^{+}$ for $i>0$. Thus from our inductive hypothesis, the code $\C(D^{i-1}(A_{i});<_{D})$ is defined if we assign the tower of subword decompositions of $A_{i}$. 
Let us denote by $\C_{m}$ the code $\C(D^{i-1}(A_{i};<_{D}))$.

 Let $<_{\textsf{res}}$ be the restriction of the ordering $<$ on $\langle \sigma_{2},\ldots,\sigma_{n-1} \rangle$. 
Under the identification of the shift map $Sh: \langle \sigma_{2},\ldots,\sigma_{n-1} \rangle \rightarrow B_{n-1}^{+}$, the ordering $<_{\textsf{res}}$ can be seen as a normal finite Thurston type ordering of $B_{n-1}^{+}$. 
Now, by giving the tower of subword decompositions of $A_{0}$, the code $\C(sh(A_{0});<_{\textsf{res}} )$ are defined.
We denote by $\C_{0}$ the code $\C(Sh(A_{0});<_{\textsf{res}})$.

Finally, let $<_{j}$ be the restriction of the ordering $<$ to $\langle \sigma_{m_{j}},\sigma_{m_{j}+1},\ldots,\sigma_{M_{j}-1} \rangle$. By identifying $\langle \sigma_{m_{j}},\sigma_{m_{j}+1},\ldots,\sigma_{M_{j}-1} \rangle$ with $B_{M_{j}-m_{j}}^{+}$ by the $(m_{j}-1)$-fold shift map $Sh^{m_{j}-1}$, we regard $<_{j}$ as a normal finite Thurston type ordering of $B_{M_{j}-m_{j}}^{+}$.
Now by giving a tower of subword decompositions, the codes $\C(Sh^{m_{j}-1}(X_{j});<_{j})$ are defined. We denote this code by $\C_{-j}$.

Now assume that we have assigned the subword decompositions of each subword $A_{i}$ and $X_{i}$ so that all of the codes $\C_{i}$ are defined. We call this iteration of subword decompositions a {\it tower of subword decompositions}.
Then we define the code $\C(W;<)$ of the positive word $W$ with the tower of subword decompositions as the sequence of codes of subwords
\[ \C(W;<) = (\ldots,\C_{i},\C_{i-1},\ldots, \C_{0},\C_{-1},\ldots,\C_{-(n-2)}).\]

This completes the definition of codes.\\

  We remark that to define the codes, it is not sufficient to choose a positive braid word. We need to indicate the tower of subword decompositions. To indicate a subword decomposition of $W$, we use the symbol $|$.
For example, To express the decomposition of $W= \sigma_{1}^{p+q}$ with $A_{-1}= \sigma_{1}^{q}$ and $A_{0}= \sigma_{1}^{p}$, we use a notation $W=\sigma_{1}^{p}|\sigma_{1}^{q}$.   
  
   Let $\Code(n,<)$ be the set of all codes of positive $n$-braid words with respect to a normal finite Thurston type ordering $<$. Next we define two lexicographical orderings $\lo$ and $\ro$ of the set $\Code(n,<)$.
To this end, we regard the code of a positive $n$-braid words as a left-infinite sequence of the code of subwords $(\ldots,\C_{i},\ldots ,\C_{-(n-2)})$, with only finitely many codes $\C_{i}$ are non-trivial. This means, for example, we regard a code $(2)$ of a 2-braid word $\sigma_{1}^{2}$ as a left infinite sequence of integers $(\ldots,0,0,2)$, not as a single integer $(2)$.

\begin{defn}[Lexicographical ordering $\lo$ and $\ro$]
Let $<$ be a normal finite Thurston type ordering.
The {\it lexicographical ordering from left} $\lo$ and the {\it lexicographical ordering from right} $\lo$ are total orderings of $\Code(n,<)$ defined by the following way.

As in the definition of codes, we begin with the case $n=2$.
For two codes $(a)$ and $(b)$ in $\Code(2,<)$, we define $(a) \lo (b)$ if and only if $a<b$, and $(a) \ro (b)$ if and only if $a<b$.

Assume that we have already defined $\lo$ and $\ro$ for all normal finite Thurston type orderings $<$ of $B_{i}$ for all $i<n$. 
 
Let 
$ \C =(\ldots,\C_{m},\ldots, \C_{-(n-2)} )$, $\C'= (\ldots,\C'_{m}, \ldots,\C'_{-(n-2)})$
be two codes in $\Code(n,*)$. 
We define $\C \lo \C'$ if 
 \[
\C_{i}=\C'_{i} \textrm{ for } i > l   \textrm{ and } \C_{l} \lo \C'_{l} \textrm{ for some } l \geq -(n-2).
\] 
Similarly, we define $\C \ro \C'$ if 
 \[
\C_{i}=\C'_{i} \textrm{ for } i < l   \textrm{ and } \C_{l} \ro \C'_{l} \textrm{ for some } l \geq -(n-2).
\] 

\end{defn}

   We are now ready to define the {\it $\C$-normal form} of a positive braids.
First of all let $<$ be a normal finite Thurston-type ordering. For a positive braid $\beta \in B_{n}^{+}$, the {\it $\C$-normal form} of $\beta$ with respect to the ordering $<$ is the positive word representative of $\beta$ which has the tower of subword decompositions whose code is maximal among the set of all codes of positive word representatives of $\beta$, with respect to the ordering $\ro$.

Since the number of positive braid word representatives of a given positive braid is finite, the $\C$-normal forms is always uniquely determined.
We denote the $\C$-normal form of a positive $n$-braid $\beta$ with respect to a normal finite Thurston type ordering $<$ by $W(\beta;<)$ or simply $W(\beta)$, and its code by $\C(\beta;<)$ or $\C(\beta)$.

 Recall that from Lemma \ref{lem:conjpos}, for conjugate left invariant total orderings $<$ and $<'$ of $B_{n}$, we can always choose a positive conjugating element $\alpha$. Using this fact, we define the $\C$-normal form for arbitrary finite Thurston type orderings.
 
\begin{defn}[The $\C$-normal form of positive braids]
Let $<$ be a finite Thurston type ordering on $B_{n}$ which is conjugate to a normal finite Thurston type ordering $<_{N}$. Let $P$ be a positive word representative of a positive conjugating element between $<$ and $<_{N}$. Fix such $<_{N}$ and $P$.

For a positive $n$-braid $\beta$, the $\C$-normal form of $\beta$ with respect to the ordering $<$ is a braid word $W(\beta;<)=W(\beta\cdot P;<_{N})\cdot P^{-1}$. The code of $\beta$, denoted by $\C(\beta;<)$, is defined by $\C(\beta;<)= \C(\beta\cdot P;<_{N})$.
\end{defn}

By definition, the set of codes $\Code(n,<)$ is a subset of $\Code(n,<_{N})$.
We define the lexicographical orderings $\lo$ and $\ro$ of $\Code(n,<)$ as the restriction of the orderings $\lo$ and $\ro$ of $\Code(n,<_{N})$ to the subset $\Code(n,<_{N})$.

For a general finite Thurston type ordering, the $\C$-normal form of a positive braid is no longer a positive braid word. Moreover, we need to fix $<_{N}$ and $P$, hence for a general orderings, the $\C$-normal form is not canonically defined.

\subsection{Examples}
   
   Before proceeding to the proof of Theorem \ref{thm:main}, we give some examples of codes and the $\C$-normal forms, which give a better understanding of the $\C$-normal form description.
 
\begin{exam}[The Dehornoy orderings of 3-braids]
\label{exam:Dehornoy3}
  
   First of all, we consider the first non-trivial case, the Dehornoy ordering of 3-braids.
   
 For a positive 3-braid word $W$, first we take a subword decomposition
\[ W = A_{m}A_{m-1}\cdots A_{0}A_{-1}\]
where $A_{i},A_{0},A_{-1} \in \langle \sigma_{2} \rangle$ for even $i$ and $A_{i} \in \langle  \sigma_{1}\rangle$ for odd $i$.
 
  The above decomposition can be explicitly written by 
\[ W= \cdots \sigma_{1}^{m_{2i-1}}\sigma_{2}^{m_{2i}}\cdots\sigma_{1}^{m_{1}}\sigma_{2}^{m_{0}}\sigma_{2}^{m_{-1}} .\] 
and the code of $W$ is given by a sequence of integers
\[ \C(W;<_{D}) = (\ldots, m_{i},m_{i-1},\ldots, m_{0},m_{-1}). \]

Now we give a concrete example.
Let $W = \sigma_{1}|\sigma_{2}|\sigma_{1}||$. Then the code of $W$ with respect to this subword decomposition ($A_{-1}=A_{0}= \varepsilon$, $A_{1}=\sigma_{1}$, $A_{2}= \sigma_{2}$, $A_{3}= \sigma_{1}$) is given by $\C(W;<_{D})= (1,1,1,0,0)$. This is not the $\C$-normal form. The $\C$-normal form of the braid $W$ is $\sigma_{2}|\sigma_{1}||\sigma_{2}$ and its code is $(1,1,0,1)$.
\end{exam}

\begin{exam}[The Dehornoy ordering of 4-braids]
\label{exam:dehornoy4}

Next we proceed to more complicated examples, the Dehornoy ordering of $B_{4}$. 
As in Example \ref{exam:Dehornoy3}, first we take a subword decomposition of a positive $4$-braid word $W$ as
\[ W = A_{m}A_{m-1}\cdots A_{1}A_{0}A_{-1}\]
 where $A_{-1},A_{0},A_{i} \in \langle \sigma_{2},\sigma_{3}\rangle$ for even $i$ and $A_{i} \in \langle  \sigma_{1},\sigma_{2}\rangle$ for odd $i$.
 
 By definition, $m_{1}=2$, $M_{1}=4$ and $m_{2}=3$, $M_{2}=4$. Thus, the subword decomposition of $A_{-1}$ is written as $A_{-1} = X_{1}X_{2}$ where $X_{1} \in \langle \sigma_{2},\sigma_{3} \rangle$ and $X_{2} \in \langle \sigma_{3} \rangle$. 
 It is easy to see the orderings $<_{1}$ and $<_{2}$ are also the Dehornoy orderings. 
Since we have already described a code of the Dehornoy ordering of $B_{3}$, we can explicitly write each code $\C_{i}$.
  
Now we give a concrete example. Let $W$ be a positive $4$-braid word with a subword decomposition
\[ W = \sigma_{1}|\sigma_{3}\sigma_{2}\sigma_{3}|\sigma_{2}^{2}\sigma_{1}^{2}||\sigma_{3}\;. \]
 That is, $A_{3}= \sigma_{1}$, $A_{2}=\sigma_{3}\sigma_{2}\sigma_{3}$, $A_{1}= \sigma_{2}^{2}\sigma_{1}^{2}$, $A_{0}= \varepsilon$, and $A_{-1}=\sigma_{3}$.
 We take a subword decomposition of $A_{-1}$ as $A_{-1} = | \sigma_{3}$.
 For a sake of simplicity, we choose the subword decompositions of each $A_{i}$ so that they have the maximal code with respect to $\ro$ among all subword decompositions of $A_{i}$.  
Then, codes $\C_{i}$ are given as follows.
\[ 
\left\{
\begin{array}{l}
\C(D^{2}(A_{3}); <_{D}) = \C(\sigma_{1}||; <_{D}) = (1,0,0)\\
\C(D^{1}(A_{2}); <_{D}) = \C(D(\sigma_{3}\sigma_{2}\sigma_{3});<_{D}) = \C(\sigma_{2}|\sigma_{1}||\sigma_{2} ; <_{D}) = (1,1,0,1) \\
 \C(D^{0}(A_{1});<_{D}) = \C(\sigma_{2}^{2}|\sigma_{1}^{2}||;<_{D}) = (2,2,0,0) \\
 \C(Sh(A_{0});<_{\textsf{res}}) = \C ( Sh(\varepsilon ) ; <_{D} ) = \C(\varepsilon ; <_{D})=(0) \\
 \C(X_{1};<_{1}) = \C(Sh(\varepsilon); <_{D}) = \C(\varepsilon ; <_{D})=(0) \\
 \C(X_{2};<_{2}) = \C(Sh(\sigma_{3}); <_{D}) = \C(\sigma_{2} ; <_{D})=(1) \\
\end{array}
\right.
\]
Consequently, the code of $W$ is given by 
\[ \C(W) = ((1,0,0),(1,1,0,1),(2,2,0,0),(0),(0),(1)).\]

The word $W$ is not the $\C$-normal form.
In fact, the $\C$-normal form of $W$ is given by $W'=\sigma_{3}^{3}|\sigma_{2}^{2}\sigma_{1}||\sigma_{2}\sigma_{3}^{2}$, and the code is given by
\[ \C(W')=((3,0,0),(2,1,0,0),(0),(1,0,0),(2)). \]

\end{exam}

\begin{exam}
\label{exam:thurston4}
Finally, we study the Thurston type ordering $<$ defined in Example \ref{exam:Ttype}, which corresponds to the permutation $(2,1,3)$. Thus, $m_{1}=1$, $M_{1}=2$ and $m_{2}=3$, $M_{2}=4$. 

 For a positive 4-braid word $W$, we take a subword decomposition of the form
\[W= A_{m}A_{m-1}\cdots A_{0}A_{-1}\]
 where $A_{i} \in \langle \sigma_{1},\sigma_{2}\rangle$ if $i$ is odd and $A_{i} \in \langle  \sigma_{2},\sigma_{3}\rangle$ if $i$ is even. Since $k=k(1)=2$, $A_{-1} \in \langle \sigma_{1},\sigma_{3}\rangle$.

First observe that the restriction of the ordering $<_{T}$ to $\langle \sigma_{2},\sigma_{3}\rangle$ is the Dehornoy ordering of $B_{3}$, under the identification by the shift map $Sh$. Thus, the code of $A_{0}$ is given as a code with respect to the Dehornoy ordering of $B_{3}$.

Next we consider a subword decomposition of $A_{-1}$. Since $m_{1}=1$, $M_{1}=2$ and $m_{2}=3$, $M_{2}=4$, the subword decomposition of $A_{-1}$ is given by $A_{-1}=X_{1}X_{2}$ where $X_{1} \in \langle \sigma_{1} \rangle$ and $X_{2} \in \langle \sigma_{3} \rangle$. 
Thus, for the word $A_{-1}=\sigma_{1}^{p}|\sigma_{3}^{q}$, $\C_{-1}=(p)$ and $\C_{-2}=(q)$.

Now we give a concrete example.
Let $W$ be the positive 4-braid word which appeared in Example \ref{exam:dehornoy4}. Take a subword decomposition 
\[ W=\sigma_{1}|\sigma_{3}\sigma_{2}\sigma_{3}\sigma_{2}^{2}|\sigma_{1}^{2}\sigma_{3}. \]
 That is, the decomposition is defined by $A_{-1}=\sigma_{1}^{2}\sigma_{3}$, $A_{0}=\sigma_{3}\sigma_{2}\sigma_{3}\sigma_{2}^{2}$, and $A_{1}=\sigma_{1}$.
 
From the above observations, each $\C(A_{i};<)$ is given as follows.
\[
\left\{
\begin{array}{l}
 \C(A_{1};<_{D})= \C(\sigma_{1}||;<_{D})= (1,0,0) \\
 \C(A_{0};<_{\textsf{res}}) = \C(Sh(\sigma_{3}|\sigma_{2}|\sigma_{3}|\sigma_{2}^{2}||);<_{D}) 
 = \C(\sigma_{2}|\sigma_{1}|\sigma_{2}|\sigma_{1}^{2}||;<_{D}) = (1,1,1,2,0,0)\\
\C(X_{1}; <_{1}) =\C(\sigma_{1}^{2};<_{D}) = (2)\\
\C(X_{2};<_{2}) = \C(Sh^{2}(\sigma_{3}); <_{D} ) = (1)
\end{array}
\right.
\]
Summarizing, the code of $W$ is given by 
\[ \C(W;<_{T}) = ((1,0,0),(1,1,1,2,0,0),(2),(1)).\]
 
This is not the $\C$-normal form. The $\C$-normal form of the braid $W$ with respect to the ordering $<_{T}$ is given by $W'=\sigma_{1}|\sigma_{3}^{3}\sigma_{2}|\sigma_{1}^{2}\sigma_{3}^{2}$, and its code is given by
\[ \C(W')= ((1,0,0),(3,1,0,0),(2),(2)). \]
\end{exam}
  
As Example \ref{exam:dehornoy4} and \ref{exam:thurston4} suggest, different orderings give completely different $\C$-normal forms.
   
\section{Combinatorial description of finite Thurston type orderings}

In this section we proof Theorem \ref{thm:main} and Theorem \ref{thm:ordertype}. Throughout this section, we use the same notation as in the Section 3. Unless otherwise stated, we always consider a normal finite Thurston type ordering $<$ which is defined by a normal curve diagram $\Gamma$ represented by the permutation $\bk=\{k(1),\ldots, k(n-1)\}$. We always put $k=k(1)$, and all the $\C$-normal forms are considered as the $\C$-normal forms with respect to the normal ordering $<$.

\subsection{Properties of $\C$-normal forms}
   
First of all we introduce a partial ordering $\succ$, which is bi-invariant under the multiplication of $B_{n}$ from both sides.

\begin{defn}
Let $W,V$ be positive $n$-braid words. We define $W \succ V$ if $W$ is obtained from $V$ by inserting positive generators $\sigma_{1},\sigma_{2},\ldots,\sigma_{n-1}$.
\end{defn}

From Property $S$, if $W \succ V$ then $W>V$ holds for all Thurston type orderings $>$. More strongly, if $W \succ V$ then $\alpha W \beta > \alpha V \beta$ holds for all $n-$braids $\alpha, \beta$ and for all Thurston type orderings $>$. Thus, we can regard the relation $\succ$ as a bi-invariant part of Thurston type orderings.

Now we begin with studying properties of the $\C$-normal form.
\begin{lem}
\label{lem:observation0}
Let $W=A_{m}A_{m-1}\cdots A_{-1}$ be a $\C$-normal form.
\begin{enumerate}
\item If $m \geq 0$ and $k \neq 1$, then $A_{0}$ is non-empty and the last letter of $A_{0}$ is $\sigma_{k}$.
\item For all $0<i\leq m$, the last letter of $A_{i}$ is $\sigma_{1}$ if $i$ is odd and
$\sigma_{n-1}$ if $i$ is even.

\end{enumerate} 
\end{lem}

\begin{proof}
We only prove (2) for odd $i$. The proofs of (1) and (2) for even $i$ are similar. 
First assume that $A_{i}$ is an empty word. Let $\sigma_{j}$ be the last letter of $A_{m}A_{m-1}\cdots A_{i+1}$. If $j \neq n-1$, then a subword decomposition of $W$ defined by $W =\cdots|\sigma_{j}|A_{i-1}|A_{i-2}|\cdots$ defines the bigger code with respect to $\ro$. If $j=n-1$, then a decomposition $\cdots |\sigma_{n-1}A_{i-1}|A_{i-2}|\cdots$ also defines the bigger code. Thus $A_{i}$ must be non-empty.
Now assume that the last letter of $A_{i}$ is not $\sigma_{1}$. Then the last letter of $A_{i}$ is $\sigma_{j}$ $(2 \leq j \leq n-2)$, so the subword decomposition $W'= \cdots A_{i+1}|A'_{i}|\sigma_{j}A_{i-1}|A_{i-2}\cdots$ defines the bigger code, which is a contradiction. 
\end{proof}

 We remark that in the case $k=1$, $A_{0}$ must be an empty word. Thus, from now on, we neglect the term $A_{0}$ if $k=1$. By Lemma \ref{lem:observation0}, we can write the $\C$-normal form in the following form.
\begin{eqnarray*}
W &= & A_{m}A_{m-1}\cdots A_{0}A_{-1}.  \\
  &= & \cdots A'_{2j+1}\sigma_{1} A'_{2j} \sigma_{n-1} \cdots A'_{0}\sigma_{k}A_{-1} \;\;(k\neq 1)\\
  &  & \cdots A'_{2j+1}\sigma_{1} A'_{2j} \sigma_{n-1} \cdots A'_{1}\sigma_{1} A_{-1} \;\;(k = 1) 
\end{eqnarray*}

Thus, the code of the $\C$-normal form must have the following form.
\[ \C(W)= (( \ldots,c_{m},\underbrace{0,\ldots, 0,0}_{*}) ,( \ldots,c_{m-1},\underbrace{0,\ldots, 0,0}_{*}),\ldots ) \]
where $c_{i}$ are non-zero codes. We call zeros $(*)$ in the code of the $\C$-normal form {\it trivial zeros}.
To consider the code of $\C$-normal forms, we do not need to consider trivial zeros, so we always neglect trivial zeros.

Next we study each word $A'_{i}$ more precisely.

\begin{lem}
\label{lem:1stobservation}
Let $W$ be the $\C$-normal form and put $A_{i},A'_{i}$ as the above.

\begin{enumerate}
\item If $m>1$ and $k \neq 1$ then $A'_{0} \succ \sigma_{2}\sigma_{3}\cdots\sigma_{k-1}$.
\item For all $0<i<m$, 
\[
\left\{
\begin{array}{l}
A'_{i} \succ \sigma_{n-2}\sigma_{n-3}\cdots\sigma_{2}   \textrm{ if } i \textrm{ is odd,}\\
A'_{i} \succ \sigma_{2}\sigma_{3}\cdots\sigma_{n-2}      \textrm{ if } i \textrm{ is even.}
\end{array}
\right.
\]
\end{enumerate} 
\end{lem}
\begin{proof}
   We only prove (2) for odd $i$. The proof of (1) and (2) for even $i$ are similar.
 Suppose $A'_{i}\not\succ \sigma_{n-2}$. Then $A'_{i} \in \langle \sigma_{1},\sigma_{2},\ldots, \sigma_{n-3} \rangle$ so $A'_{i}$ commutes with $\sigma_{n-1}$. By Lemma \ref{lem:observation0}, the last letter of $A_{i+1}$ is $\sigma_{n-1}$. Thus, we obtain a new word decomposition 
\[ W= \cdots A'_{i+1}|A'_{i}\sigma_{1}| \sigma_{n-1}A_{i-1}|A_{i-2}\cdots, \] 
which defines the bigger code. Thus, $A'_{i} \succ \sigma_{n-2}$.
 Now let us write $A'_{i}=V \sigma_{n-2}V'$, where $V' \in \langle \sigma_{1},\sigma_{2},\ldots,\sigma_{n-3}\rangle$.
  By the similar arguments, we obtain $V' \succ \sigma_{n-3}$. Iterating this argument, we conclude that $A'_{i} \succ \sigma_{n-2}\sigma_{n-1}\cdots \sigma_{2}$.

\end{proof}

   In some special cases, the above subword estimation results can be improved.
\begin{lem}
\label{lem:subword}
Let $W=A_{m}A_{m-1}\cdots A_{0} A_{-1}$ be a $\C$-normal form and $i>0$.

\begin{enumerate}
\item If $m > 0$, $k\neq 1$ and the last letter of $A'_{1}$ is not $\sigma_{1}$, then 
\[ 
A'_{0} \succ (\sigma_{3}\sigma_{2}) (\sigma_{4}\sigma_{3}) \cdots (\sigma_{j+1}\sigma_{j}) \sigma_{j}\sigma_{j+1}\cdots\sigma_{k-1}\;\;\; \textrm{for some}j<k. 
\]
\item If $i<m$ is odd and the last letter of $A'_{i+1}$ is not $\sigma_{n-1}$, then
\[
 A'_{i} \succ (\sigma_{n-3}\sigma_{n-2}) (\sigma_{n-4}\sigma_{n-3})\cdots (\sigma_{j-1}\sigma_{j}) \sigma_{j}\sigma_{j-1}\cdots\sigma_{2} \;\;\;\textrm{for some}j>2.
\]
\item If $i<m$ is even and the last letter of $A'_{i+1}$ is not $\sigma_{1}$, then
\[ 
A'_{i} \succ (\sigma_{3}\sigma_{2})(\sigma_{4}\sigma_{3})\cdots(\sigma_{j+1}\sigma_{j})\sigma_{j}\sigma_{j+1}\cdots \sigma_{n-2} \;\;\textrm{for some}j<n-2.
\]
\end{enumerate} 
\end{lem}

\begin{proof}
   We only prove (3). Other cases are similar. 
From Lemma \ref{lem:1stobservation}, we already know that $A_{i} \succ \sigma_{2}\sigma_{3}\cdots\sigma_{n-2}$. Thus, the word $A'_{i}$ is of the form 
\[A'_{i} = V_{1}\sigma_{2}V_{2}\sigma_{3}\cdots V_{n-3}\sigma_{n-2}, \]
 where $V_{j} \in \langle \sigma_{j+1},\sigma_{j+2},\cdots, \sigma_{n-1}\rangle $, possibly empty.

Put $A'_{i+1}=A''_{i+1}\sigma_{j}$. 
First we observe $j=2$. Because otherwise, the last letter of $A'_{i+1}$ commutes with $\sigma_{1}$, so we obtain a subword decomposition 
\[ W' = \cdots |A''_{i+1}\sigma_{1}| \sigma_{j}A_{i}|A_{i-1}\cdots \]
 which defines the bigger code.
Thus, the word $W$ is written as
\[ W = \cdots A_{i+1}''\sigma_{2}\sigma_{1}|V_{1}\sigma_{2}V_{2}\cdots V_{n-3}\sigma_{n-2}V_{n-2}\sigma_{n-1}|A_{i-1} \cdots. \]

 If $V_{1} \not\succ \sigma_{3}$, then $V_{1}$ contains neither $\sigma_{1}$ nor $\sigma_{3}$, so $V_{1}$ commutes with $\sigma_{2}$. Since $\sigma_{1}$ commutes with each $V_{i}$ for $i>1$, we can change $W$ as
\begin{eqnarray*}
W & = & \cdots |A_{i+1}''\sigma_{2}\sigma_{1}\sigma_{2}V_{1}V_{2}\cdots V_{n-3}\sigma_{n-2}V_{n-2}\sigma_{n-1}|A_{i-1}|\cdots \\
  & = & \cdots |A_{i+1}''\sigma_{1}\sigma_{2}\sigma_{1}V_{1}V_{2}\cdots V_{n-3}\sigma_{n-2}V_{n-2} \sigma_{n-1}|A_{i-1}|\cdots \\
  & = & \cdots |A_{i+1}'' V'_{1}\sigma_{1}\sigma_{2}\sigma_{1}V_{2}\cdots V_{n-3} \sigma_{n-2}V_{n-2}\sigma_{n-1}|A_{i-1}| \cdots \\
  & = & \cdots |A_{i+1}'' V'_{1}\sigma_{1}\sigma_{2}V_{2}\cdots V_{n-3} \sigma_{n-2}V_{n-2}\sigma_{n-1}| \sigma_{1}A_{i-1}| \cdots
\end{eqnarray*}
 by using the braid relations. Here the third equality is obtained by applying the relations $(\sigma_{1}\sigma_{2}\sigma_{1})\sigma_{1} =\sigma_{2}(\sigma_{1}\sigma_{2}\sigma_{1})$ and $(\sigma_{1}\sigma_{2}\sigma_{1})\sigma_{j} = \sigma_{j}(\sigma_{1}\sigma_{2}\sigma_{1})$ $(j \geq 4)$ repeatedly.
 
Then the last subword decomposition defines the bigger code, which is contradiction.
Thus $V_{1} \succ \sigma_{3}$.
If $V_{1} \succ \sigma_{3}\sigma_{2}$, we are done. Otherwise, by iterating the similar arguments for each $V_{i}$, we obtain the desired result.

\end{proof}

For a $\C$-normal form $W= A_{m}A_{m-1}\cdots A_{0}A_{-1}$, we define the integer $a_{i},b_{i}$ and $c_{i}$ as follows. For odd $0<i<m$, we define $a_{i} = 2$ if the last letter of $A'_{i}$ is $\sigma_{1}$ and $a_{i}=1$ otherwise. Similarly, for even $i$, we define $a_{i} = 2$ if the last letter of $A'_{i}$ is $\sigma_{n-1}$ and $a_{i}=1$ otherwise. 

By definition, $a_{i+1}=1$ if and only if the assumption of Lemma \ref{lem:subword} is satisfied for $i$. Therefore if $a_{i+1}=1$, then 
\[
\left\{
\begin{array}{l}
A_{0} \succ (\sigma_{3}\sigma_{2})(\sigma_{4}\sigma_{3})\cdots(\sigma_{j+1}\sigma_{j})\sigma_{j}\sigma_{j+1}\cdots\sigma_{k}\\
A_{i} \succ (\sigma_{3}\sigma_{2})(\sigma_{4}\sigma_{3})\cdots(\sigma_{j+1}\sigma_{j})\sigma_{j}\sigma_{j+1}\cdots\sigma_{n-2}\sigma_{n-1}^{a_{i}} \;\;(i:even)\\
A_{i} \succ (\sigma_{n-3}\sigma_{n-2})(\sigma_{n-4}\sigma_{n-3})\cdots(\sigma_{j-1}\sigma_{j})\sigma_{j}\cdots\sigma_{2}\sigma_{1}^{a_{i}} \;\;(i:odd)
\end{array}
\right.
\]
holds for some $j$. For even $i$, let us define an integer $b_{i}$ $(3 \leq b_{i} \leq n-1)$ as the maximal integer satisfying
\[
\left\{
\begin{array}{l}
A_{0} \succ (\sigma_{b_{0}} \sigma_{b_{0}-1}\cdots \sigma_{4}) (\sigma_{3}\sigma_{2})(\sigma_{4}\sigma_{3})\cdots(\sigma_{j+1}\sigma_{j})\sigma_{j}\sigma_{j+1}\cdots\sigma_{k} \\
A_{i} \succ  (\sigma_{b_{i}} \sigma_{b_{i}-1}\cdots \sigma_{4}) (\sigma_{3}\sigma_{2})(\sigma_{4}\sigma_{3})\cdots(\sigma_{j+1}\sigma_{j})\sigma_{j}\sigma_{j+1}\cdots\sigma_{n-1}^{a_{i}}. \\
\end{array}
\right.
\]
For odd $i$, we define $b_{i}$ $(1 \leq b_{i} \leq n-3)$ as the minimal integer satisfying 
\[
A_{i} \succ  (\sigma_{b_{i}} \sigma_{b_{i}+1}\cdots \sigma_{n-4}) (\sigma_{n-3}\sigma_{n-2})(\sigma_{n-4}\sigma_{n-3})\cdots(\sigma_{j-1}\sigma_{j})\sigma_{j}\sigma_{j-1}\cdots\sigma_{1}^{a_{i}}. \]

We will consider, for example when $b_{0}=3$, the above formula simply means
$ A_{0} \succ (\sigma_{3}\sigma_{2})(\sigma_{4}\sigma_{3})\cdots(\sigma_{j+1}\sigma_{j})\sigma_{j}\sigma_{j+1}\cdots\sigma_{k}$.

We define $b_{i}=0$ if $a_{i+1}=1$.
Now we show that if $a_{i+1}=1$, then the subword estimation given in Lemma \ref{lem:1stobservation} for $A_{i+1}$ can also be improved.

\begin{lem}
\label{lem:subword2}
If $a_{i+1}=1$, then there exists an integer $c_{i}$ such that 
\[
\left\{
\begin{array}{l}
A'_{i+1} \succ \sigma_{n-2}\cdots \sigma_{c_{i}+1}\sigma_{c_{i}}^{2}\sigma_{c_{i}-1}\cdots \sigma_{2} \;\;(i: even, \; c_{i}<b_{i})\\
A'_{i+1} \succ \sigma_{2}\cdots \sigma_{c_{i}-1}\sigma_{c_{i}}^{2}\sigma_{c_{i}+1}\cdots \sigma_{n-2} \;\;(i:odd, \; c_{i}>b_{i}) 
\end{array}
\right.
\]
holds.
\end{lem}

\begin{proof}

We prove the lemma for even $i$. The proof for odd $i$ is similar. 
As in the proof of Lemma \ref{lem:subword}, first let us denote $A'_{i}$ as 
\[ A'_{i} = V_{1}\sigma_{2}V_{2}\sigma_{3}\cdots V_{n-3}\sigma_{n-2} \]
where $V_{j} \in \langle \sigma_{j+1},\sigma_{j+2},\ldots,\sigma_{n-2} \rangle$. From the definition of $b_{i}$, we can write the word $V_{1}$ as
\[ V_{1}=  W_{b_{i}}\sigma_{b_{i}} W_{b_{i-1}} \sigma_{b_{i}-1}\cdots W_{4}\sigma_{4} W_{3} \sigma_{3} \]
 where $W_{j} \in \langle \sigma_{2},\ldots, \sigma_{j}\rangle$.
 
Assume that $A'_{i+1} \not \succ \sigma_{n-2}\cdots \sigma_{c_{i}+1}\sigma_{c_{i}}^{2}\sigma_{c_{i-1}}\cdots \sigma_{2} $ for all $c_{i} <b_{i}$. 
Then from the assumption that $W$ is the $\C$-normal form, the word $A_{i+1}$ is written as 
\[ A_{i+1} = \cdots (\sigma_{b_{i}}\sigma_{b_{i}-1}\cdots\sigma_{3} \sigma_{2}\sigma_{1}).\]
Because otherwise, we can produce a new word representative and subword decomposition which defines a bigger code.

We prove the lemma by induction of $b_{i}$ and $l(W_{b_{i}})$, the length of the word $W_{b_{i}}$.
First we consider the case $b_{i}=3$ and $l(W_{3})=0$. In this case, We can change the word $A_{i+1}A_{i}$ as 
\[
A_{i+1}A_{i}  = \cdots \sigma_{3}\sigma_{2}\sigma_{1}|\sigma_{3}\sigma_{2} V_{2}\cdots =  \cdots \sigma_{2}\sigma_{3}(\sigma_{1}\sigma_{2}\sigma_{1})V_{2} \cdots. 
\]
Hence, by the same argument in the proof of Lemma \ref{lem:subword}, from the existence of the subword $(\sigma_{1}\sigma_{2}\sigma_{1})$ we can construct a new word representative and subword decomposition which defines the bigger code, by pushing out $\sigma_{1}$. Hence it is contradiction.

Next we consider a general case. If $l(W_{b_{i}}) = 0$, then we can change the word $A_{i+1}A_{i}$ as  
\begin{eqnarray*}
A_{i+1}A_{i}& = &\cdots \sigma_{b_{i}}\sigma_{b_{i}-1}\cdots \sigma_{2} \sigma_{1}|\sigma_{b_{i}}W_{b_{i-1}}\sigma_{b_{i}-1}\cdots \\
 & = & \cdots \sigma_{b_{i}-1} (\sigma_{b_{i}}\sigma_{b_{i}-1}\cdots \sigma_{2}\sigma_{1})W_{b_{i-1}}\sigma_{b_{i}-1}\cdots.
\end{eqnarray*}

If$l(W_{b_{i}}) \geq 1$, let us put $W_{b_{i}}= \sigma_{p} W'_{b_{i}}$. Then we can change the word $A_{i+1}A_{i}$ as
\begin{eqnarray*}
A_{i+1}A_{i} & = &\cdots (\sigma_{b_{i}} \cdots \sigma_{p}\sigma_{p-1}\cdots \sigma_{1})\sigma_{p}W'_{b_{i}} \sigma_{b_{i}}W_{b_{i-1}}\cdots \\
 & = & \cdots( \sigma_{b_{i}} \cdots \sigma_{p-1} \sigma_{p}\sigma_{p-1}\cdots \sigma_{1}) W'_{b_{i}}\sigma_{b_{i}} W_{b_{i}-1}\cdots\\
 & = & \cdots \sigma_{p-1}( \sigma_{b_{i}}\sigma_{b_{i}-1}\cdots \sigma_{1}) W'_{b_{i}}\sigma_{b_{i}} W_{b_{i}-1}\cdots.
\end{eqnarray*}
Hence in both cases, by inductive hypothesis, we obtain a new word representative and subword decomposition which defines a bigger code. This is a contradiction.
\end{proof}

Finally, we define $c_{i}=0$ if $a_{i+1}=2$. This completes the definition of $a_{i},b_{i}$ and $c_{i}$. 

Now we construct two positive braid words $\underline{W}$ and $\overline{W}$ for each $\C$-normal form $W$.
These two words play an important role in the proof of Theorem \ref{thm:main}.
First of all, we define the word $\underline{A_{i}}$.

If $a_{i+1}=1$, then we define
\[
\left\{
\begin{array}{l}
\underline{A_{0}} = (\sigma_{b_{0}} \sigma_{b_{0}-1}\cdots \sigma_{4}) (\sigma_{3}\sigma_{2})(\sigma_{4}\sigma_{3})\cdots(\sigma_{j+1}\sigma_{j})\sigma_{j}\sigma_{j+1}\cdots\sigma_{k}\\
\underline{A_{i}} = (\sigma_{b_{i}} \sigma_{b_{i}-1}\cdots \sigma_{4}) (\sigma_{3}\sigma_{2})(\sigma_{4}\sigma_{3})\cdots(\sigma_{j+1}\sigma_{j})\\
 \hspace{3cm} (\sigma_{j} \cdots\sigma_{c_{i}-1}\sigma_{c_{i}}^{2}\sigma_{c_{i}+1}\cdots \sigma_{n-2}\sigma_{n-1}^{a_{i}}) \hspace{1cm}(i:even)\\ 
\underline{A_{i}} =   (\sigma_{b_{i}} \sigma_{b_{i}+1}\cdots \sigma_{n-4}) (\sigma_{n-3}\sigma_{n-2})(\sigma_{n-4}\sigma_{n-3})\cdots(\sigma_{j-1}\sigma_{j}) \\
\hspace{3cm}(\sigma_{j} \cdots\sigma_{c_{i}+1}\sigma_{c_{i}}^{2}\sigma_{c_{i}-1}\cdots \sigma_{2}\sigma_{1}^{a_{i}})
\hspace{1cm}(i:odd)
\end{array}
\right.
\]

We remark that, for example, if $c_{i} > j$ for even $i$, the above definition simply implies 
\[ \underline{A_{i}} = (\sigma_{b_{i}} \sigma_{b_{i}-1}\cdots \sigma_{4}) (\sigma_{3}\sigma_{2})(\sigma_{4}\sigma_{3})\cdots(\sigma_{j+1}\sigma_{j})\sigma_{j} \cdots  \sigma_{n-2}\sigma_{n-1}^{a_{i}}.  \]

If $a_{i+1}=2$, then we define
\[
\left\{
\begin{array}{l}
\underline{A_{0}} = (\sigma_{2}\sigma_{3}\cdots \sigma_{k})\\
\underline{A_{i}} = (\sigma_{2}\cdots \sigma_{c_{i}-1}\sigma_{c_{i}}^{2} \sigma_{c_{i}+1}\cdots \sigma_{n-1}^{a_{i}}) \;\;\;(i:even)\\ 
\underline{A_{i}} = (\sigma_{n-2}\sigma_{n-3}\cdots \sigma_{c_{i}+1}\sigma_{c_{i}}^{2} \sigma_{c_{i}-1}\cdots \sigma_{1}^{a_{i}}) \;\;\;(i:odd)
\end{array}
\right.
\]

By Lemma \ref{lem:1stobservation}, \ref{lem:subword}, and \ref{lem:subword2}, in either case $A_{i} \succ \underline{A_{i}}$ holds for all $i$.
Using the words $\underline{A_{i}}$, we define the braid word $\underline{W}$ by
\[\underline{W} = A_{m}\underline{A_{m-1}}\, \underline{A_{m-2}}\cdots \underline{A_{0}}A_{-1}.\]

Let $d_{0}$ and $d_{1}$ be $n$-braids defined by 
\[ d_{0} = (\sigma_{2}\sigma_{3}\cdots\sigma_{n-1})^{n-1}, d_{1}=(\sigma_{1}\sigma_{2}\cdots\sigma_{n-2})^{n-1}. \] 
We define the braid word $\overline{W}$ by
\[ \overline{W} = A_{m}d_{[m-1]}^{p} d_{[m-2]}^{p} \cdots d_{[0]}^{p} A_{-1} \] 
 where $p$ is a sufficiently large integer, and $[i]=0$ if $i$ is even and $[i]=1$ if $i$ is odd.
By taking a sufficiently large integer $p$, $\overline{W}$ is obtained by inserting positive generators $\{\sigma_{i}\}$ into the word $W$. Thus, from Property $S$, $\overline{W} \succ W \succ \underline{W}$ holds.

   Our idea to prove Theorem \ref{thm:main} is that, for two $\C$-normal forms $W$ and $V$, we compare $\underline{W}$ and $\overline{V}$ instead of comparing $V$ and $W$ directly. 
   An idea to insert or to delete the generators for a given braid word so that the action of the braid become simpler and is easier to compare had already appeared in the author's previous paper \cite{i1}, to estimate the Dehornoy floor. The arguments given belows are  straightforward generalization and refinements of the arguments given \cite{i1}.
 
By definition, the most parts of the words $\overline{W}$ and $\underline{V}$ are explicitly given, we can ``draw" the image of the curve diagram $\Gamma$ under their braid actions for initial parts, which is enough to compare $\overline{W}$ and $\underline{V}$. 
 The action of the braids $d_{0}$ and $d_{1}$ on the punctured disc $D_{n}$ are easy to describe. The braid $d_{0}$ corresponds to the full Dehn-twist along the circle enclosing the punctures $\{p_{1},p_{2},\ldots,p_{n-1}\}$ and the braid $d_{1}$ corresponds to the full Dehn-twist along the circle enclosing the punctures $\{p_{2},p_{3},\ldots,p_{n}\}$.
 
Similarly, we can see actions of $\underline{A_{i}}$, as we will explain the next section.

\subsection{Cutting sequence presentation and computations}
 
 Before stating the actions of the braids $\underline{A_{i}}$ and $\underline{W}$, we introduce a cutting sequence presentation of embedded arcs. 
This is a method of encoding an embedded curve into a sequence of signed integers, introduced in \cite{fgrrw}. 

  Let $\Sigma$ be a curve diagram consists of the horizontal arcs which connect the puncture points (see Figure \ref{fig:cutting}). For a properly embedded oriented arc $C$ in $D_{n}$ which transversely intersects $\Sigma$, we associate the sequence of signed integers as the following manner.

 Let $q_{i}$ be the $i$-th intersection point of $C$ with $\Sigma$. We define $C(i) = +k$ (resp. $-k$) if $q_{i}$ lies on $\Sigma_{k}$ and the sign of the intersection at $q_{i}$ is positive (resp. negative). Since the complement of the curve diagram $\Sigma$ is a disc, the isotopy class of $C$ can be uniquely determined by the finite sequence of signed integers $( C(1), C(2),\ldots , C(m))$.
 
We call this sequence the {\it cutting sequence presentation} of the arc $C$.
In a cutting sequence presentation, an existence of a bigon between $C$ and $\Sigma$ corresponds to an existence of the subsequence of the form $(\pm i, \mp i)$. Therefore $C$ is tight to $\Sigma$ if and only if the cutting sequence presentation of $C$ contains no subsequences of form $(\pm i, \mp i)$. Removing a bigon corresponds to removing the subsequence $(\pm i, \mp i)$.
We call the cutting sequence presentation is {\it tight} if the corresponding arc is tight to $\Sigma$.

\begin{figure}[htbp]
 \begin{center}
\includegraphics[width=70mm]{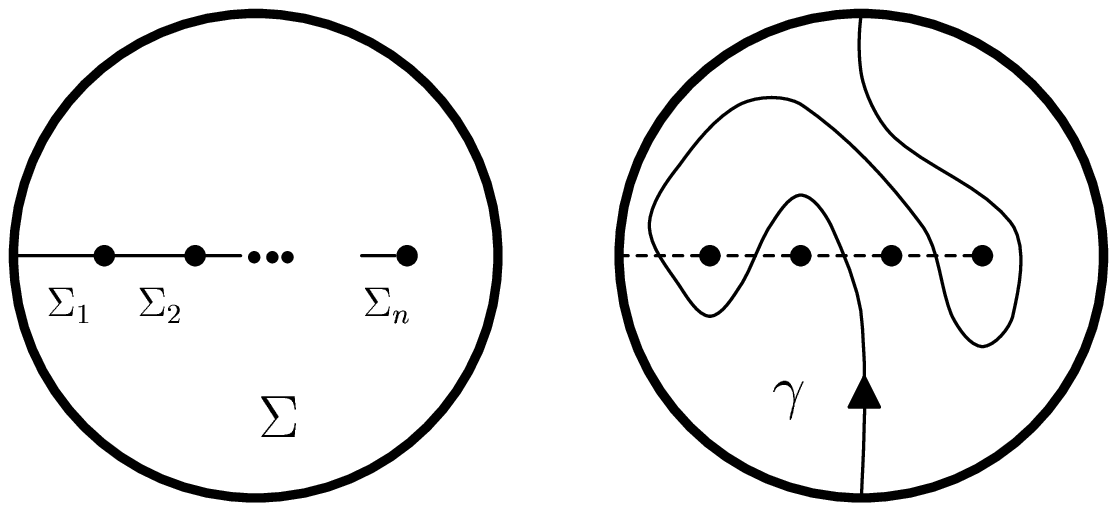}
 \end{center}
 \caption{Cutting sequence presentation}
 \label{fig:cutting}
\end{figure}
\begin{exam}
In the right figure in Figure \ref{fig:cutting}, the cutting sequence presentation of the arc $\gamma$ is $(+3,-2,+1,-4)$. As is easily observed, this is the tight cutting sequence presentation.
\end{exam}

 The actions of the braid groups on cutting sequence presentations are easy to describe.
 Since the braid $\sigma_{i-1}$ corresponds to the half Dehn-twist along the arc $\Sigma_{i}$, the cutting sequence presentation of the arc $\sigma_{i-1}(C)$ is obtained from the original cutting sequence of $C$ by replacing a subsequence $(+i)$ (resp. $(-i)$) with $(+(i-1),-i,+(i+1))$ (resp. $(-(i+1),+i,-(i-1))\,$).
The obtained cutting sequence presentation might be non-tight and depends on a particular choice of word representative of a braid, even if the original cutting sequence presentation is tight. Therefore, sometimes the tight cutting sequence presentation of an arc might change dramatically. However, in some cases we can obtain an initial segment of the tight cutting sequence presentation of an arc as the following lemma shows.
 
\begin{lem}
\label{lem:cutcalc}
Let $\beta$ be a positive $n$-braid and $\Gamma$ be a properly embedded arc. Let us denote the tight cutting sequence presentation of $\beta(\Gamma)$ by $(C(1),\ldots,C(k),C(k+1),\ldots)$ and assume that $|C(j)| \neq i$ for all $j<k$.

\begin{enumerate}
\item If $C(k)= +i$ and $C(k-1)\neq -(i-1)$, then the tight cutting sequence presentation of the braid  $\sigma_{i-1}\beta$ is 
\[ (C(1),C(2),\ldots, C(k-1), +(i-1),-i,\ldots ).\]
\item If $C(k)= -i$ and $C(k-1)\neq +(i-1)$, then the tight cutting sequence presentation of the braid $\sigma_{i-1}\beta$ is 
\[ (C(1),C(2),\ldots, C(k-1),-(i+1),+i,\ldots ).\]
\end{enumerate}
 
\end{lem}
\begin{proof}
We prove (1). The proof of (2) is similar.
By a direct computation, the cutting sequence presentation of $\sigma_{i}\beta(\Gamma)$ is given by 
\[ (C(1),C(2),\ldots, C(k-1), +(i-1),-i,+i+1, \ldots ). \] 

Since the original cutting sequence is tight, this sequence does not contain a subsequence of form $(\pm a, \pm b, \mp b,\mp a)$. This means, under the process of making the sequence tight, $C(j)$ disappear if and only if $C(j \pm 1) = -C(j)$. From the assumption of $C(k-1) \neq -(i-1)$, the term $+(i-1)$ does not disappear, so the tight cutting sequence presentation of $\sigma_{i}\beta(\Gamma)$ is given by  
\[ (C(1),C(2),\ldots, C(k-1), +(i-1),-i,\ldots ).\]
\end{proof}

  By using Lemma \ref{lem:cutcalc}, we can calculate an initial segment of the tight cutting sequence presentation of the following special braids which are related to $\overline{W}$ and $\underline{W}$.

\begin{prop}
\label{prop:cutcalc}

 Let $W = A_{m}A_{m-1}\cdots A_{-1}$ be the $\C$-normal form of a positive $n$-braid $\beta$ and 
\[ \overline{W}= A_{m}d_{[m-1]}^{p} d_{[m-2]}^{p} \cdots d_{[0]}^{p} A_{-1},\; \underline{W}= A_{m}  \underline{A_{m-1}}\, \underline{A_{m-2}}\cdots \underline{A_{0}}A_{-1}.\]
be positive braid words defined in the previous section.
\begin{description}
\item[($A$)] The tight cutting sequence presentation of the arc \\
 $[\sigma_{1}^{a}d_{[m-1]}^{p} \cdots d_{[0]}^{p} A_{-1} ](\Gamma_{1}) $ for odd $m$ is
\[
\left\{
\begin{array}{l}
 ( \underbrace{ +1,+1,\ldots , +1}_{(m+1)\slash 2},-2,+3,\ldots ) \;\;\; (a=1)\\
 ( \underbrace{ +1,+1,\ldots , +1}_{(m+1)\slash 2},-3,+2,\ldots ) \;\;\; (a=2)\\
\end{array}
\right.
\]

\item[($A'$)] The tight cutting sequence presentation of the arc  \\
$ [\sigma_{n-1}^{a} d_{[m-1]}^{p} d_{[m-2]}^{p} \cdots d_{[0]}^{p} A_{-1} ](\Gamma_{1}) $ for even $m$ is
\[
\left\{
\begin{array}{l}
 ( \underbrace{ +1,+1,\ldots , +1}_{m\slash 2},+n,-(n-1),\ldots ) \;\;\; (a=1)\\
 ( \underbrace{ +1,+1,\ldots , +1}_{m\slash 2},+(n-1),-n,\ldots ) \;\;\; (a=2)\\
\end{array}
\right.
\]

\item[($B$)] The tight cutting sequence presentation of the arc \\
$ [\sigma_{1}^{a_{m}}\underline{A_{m-1}}\, \underline{A_{m-2}}\cdots \underline{A_{0}}A_{-1}](\Gamma_{1}) $ for odd $m$ is
\[
\left\{
\begin{array}{l}
( \underbrace{ +1,+1,\ldots, +1 }_{(m+1)\slash 2},-2,+3,\ldots )  \textrm{   if } a_{m}=1 \\
( \underbrace{ +1,+1,\ldots, +1 }_{(m+1)\slash 2},-3,+2,\ldots ) \textrm{   if } a_{m}=2.
\end{array}
\right.
\]

\item[($B'$)] The tight cutting sequence presentation of the arc \\
 $ [\sigma_{n-1}^{a_{m}}\underline{A_{m-1}}\, \underline{A_{m-2}}\cdots \underline{A_{0}}A_{-1}](\Gamma_{1}) $ for even $m$ is
\[
\left\{
\begin{array}{l}
( \underbrace{ +1,+1,\ldots, +1 }_{m\slash 2},+n,-(n-1),\ldots ) \textrm{   if } a_{m}=1 \\
( \underbrace{ +1,+1,\ldots, +1 }_{m\slash 2},+(n-1),-n,\ldots ) \textrm{   if } a_{m}=2.
\end{array}
\right.
\]

\end{description}
\end{prop}

\begin{proof}
 The above calculations are confirmed by a combination of the following sub-calculations, which are easily observed from Lemma \ref{lem:cutcalc}.
 
\begin{itemize}
\item If $\Gamma$ has the tight cutting sequence $(*,+(s+1),\ldots)$ where $*$ is a subsequence which contains no $\pm (s+1), \ldots, \pm (q+1)$, then the tight cutting sequence of $(\sigma_{q}\sigma_{q+1}\cdots \sigma_{s})(\Gamma)$ is $(*,+q,-(q+1),\ldots)$.

\item If $\Gamma$ has the tight cutting sequence $(*,-(s+1),\ldots)$ where $*$ is a subsequence which contains no $\pm (q+1),\ldots ,\pm (s+1)$, then the tight cutting sequence of $(\sigma_{q}\sigma_{q-1}\cdots \sigma_{s})(\Gamma)$ is $(*,-(q+2),+(q+1),\ldots)$.

\item If $\Gamma$ has the tight cutting sequence $(*,+q,-(q+1),\cdots)$, where $*$ is a subsequence which contains no $\pm (q+1), \pm (q+2)$, then the tight cutting sequence of $(\sigma_{q+1}\sigma_{q})\Gamma$ is 
$ (*, +q, -(q+3),\ldots )$.

\item If $\Gamma$ has the tight cutting sequence $(*,-(q+2),+(q+1),\ldots)$ where $*$ is a subsequence which contains no $\pm q, \pm (q+1)$, then the tight cutting sequence of $ (\sigma_{q-1}\sigma_{q})\Gamma$ is $(*,-(q+2),+(q-1),\ldots)$.

\item If $\Gamma$ has the tight cutting sequence $(*,+(q+1),-r,\ldots)$ where $r \neq -(q+2)$ and $*$ is a subsequence which contains no $\pm (q+1), \pm (q+2)$, then the tight cutting sequence of $ (\sigma_{q+1}\sigma_{q})\Gamma$ is $(*,+q,-(q+2),\ldots)$.

\item If $\Gamma$ has the tight cutting sequence $(*,-(q+1),r,\ldots)$ where $r \neq +q$ and $*$ is a subsequence which contains no $\pm q,\pm (q+1)$, then the tight cutting sequence of $ (\sigma_{q-1}\sigma_{q})\Gamma$ is $(*,-(q+2),+q,\ldots)$.
\end{itemize}
\end{proof}

\subsection{Proof of Theorem \ref{thm:main} and \ref{thm:ordertype}}
 
 Now we are ready to prove our main theorems.
 
\begin{proof}[Proof of Theorem \ref{thm:main}]
  
  First we remark that from the definition of the $\C$-normal forms and codes, we only need to prove the theorem for normal finite Thurston type orderings.

   We prove the theorem by induction of $n$. The case $n=2$ is trivial. Assume that the assertion is proved for all normal finite Thurston type orderings of $B_{i}$ for all $i<n$.
   
   Let $<$ be a normal finite Thurston type ordering on $B_{n}$ define by the normal curve diagram $\Gamma$. For positive $n$-braids $\alpha$ and $\beta$, assume that $\C(\alpha;<) \lo \C(\beta;<)$ holds. We denote the $\C$-normal form of $\beta$ by $W = A_{m}A_{m-1}\cdots A_{-1}$ and the $\C$-normal form of $\alpha$ by $V=C_{l}C_{l-1}\cdots C_{-1}$. We prove $\alpha < \beta$ by showing $\overline{V} < \underline{W}$.
   
   Since the ordering $<$ is left-invariant, we can always assume that the initial letters of $W(\alpha;<)$ and $W(\beta;<)$ are different, by annihilating the common prefixes.
Let us put
\[ \underline{W} = \underline{W(\beta;<)} = A_{m}\underline{A_{m-1}}\, \underline{A_{m-2}}\cdots \underline{A_{0}}A_{-1} \]
and
\[ \overline{V} = \overline{W(\alpha;<)} = C_{l}d_{[l-1]}^{p} d_{[l-2]}^{p} \cdots d_{[0]}^{p} C_{-1}. \]

Since $\C(\alpha;<) \lo \C(\beta;<)$, the inequality $l\leq m$ holds. If $l \neq m$, then by Proposition \ref{prop:cutcalc}, we conclude the arc
 $\underline{W}(\Gamma_{1})$ moves more left than $\overline{V}(\Gamma_{1})$, so $\alpha<\overline{V}<\underline{W}<\beta$. Therefore we only need to consider the case $m=l$. 

First we consider the case $m=0$ or $-1$.
Assume that $m=-1$, thus $\underline{W} = A_{-1}$ and $\overline{V}=C_{-1}$.
Let us denote the decompositions of $A_{-1}$ and $C_{-1}$ as $A_{-1}=X_{1}\cdots X_{n-2}$ and $ C_{-1}= Y_{1}\cdots Y_{n-2}$ respectively.
Since we have already assumed that the initial letters of $W(\alpha)$ and $W(\beta)$ are different, so 
our assumption $C_{-1} \lo A_{-1}$ implies $X_{j}=Y_{j}=\varepsilon$ for $j=1,\cdots,i-1$ and
$\C_{-i}  \ro \C_{-i}'$.

Let us write the decomposition of $X_{i}$ as 
\[ X_{i}= P_{m} \cdots P_{0}P_{-1}.\]
The assumption that $A_{-1}$ is the $\C$-normal form implies $P_{-1}$ must be the empty word.
Because otherwise,
 by regarding $P_{-1}X_{i-1}$ as a new $X_{i-1}$, we obtain a new subword decomposition which defines the bigger code.
Similarly, for the decomposition $Y_{i}=Q_{m}\cdots Q_{0}Q_{-1}$ of $Y_{i}$, the last subword $Q_{-1}$ is the empty word.

Now, the fact both $P_{-1}$ and $Q_{-1}$ are empty words implies that both the braids $A_{-1}$ and $C_{-1}$ move the arc $\Gamma_{i}$. Therefore the comparison of $A_{-1}$ and $C_{-1}$ can be done by seeing the image of $\Gamma_{i}$.
Recall that we have defined the code of $X_{i}$ and $Y_{i}$ so that they coincide with the code with respect to the restriction of the ordering to $\langle \sigma_{m_{i}},\sigma_{m_{i}+1},\ldots,\sigma_{M_{i}}\rangle$.
Now by inductive hypothesis, we conclude that $X_{i}(\Gamma_{i})$ moves left side of$Y_{i}(\Gamma_{i})$, hence $\alpha < \beta$.
The case $m=0$ is similar.

 Next we consider the case $m>0$. We only show $m$ is odd case. The even $m$ case is similar.
 From Proposition \ref{prop:cutcalc}, the tight cutting sequence presentation of the arc 
\[ [\sigma_{1}^{a_{m}}\underline{A_{m-1}}\, \underline{A_{m-2}}\cdots \underline{A_{0}}A_{-1}](\Gamma_{1}) \] 
 have the initial segment
\[( \underbrace{+1,+1,\ldots, +1}_{(m-1) \slash 2} ,\underbrace{+1,-2,+3}_{(*1)} ).\]
if $a_{m}=1$ and
\[( \underbrace{+1,+1,\ldots, +1}_{(m-1) \slash 2} ,\underbrace{+1,-3,+2}_{(*2)} ).\]
if $a_{m}=2$.

 From the definition of $\C$-normal form, both $A_{m}$ and $C_{m}$ are also the $\C$-normal form with respect to the Dehornoy ordering of $B_{n-1}$. 
 
Let $\Pi$ be the normal curve diagram in Example \ref{exam:Dtype}, which defines the Dehornoy ordering. Since $\C(V;<) \lo \C(W;<)$, the inequality $\C(A_{m};<_{D}) \lo \C(C_{m};<_{D})$ holds. Therefore by the inductive hypothesis, the arc $A_{m}(\Pi_{1})$ moves left side of $C_{m}(\Pi_{1})$. In particular, $\underline{A_{m}}(\Pi_{1})$ moves more left than $\overline{C_{m}}(\Pi_{1})$.

Now let us observe the tight cutting sequence presentation of the arc $\sigma_{1}(\Pi_{1})$ is $(+1,-2,+3)$, which is identical with the subsequence $(*1)$, and the tight cutting sequence of $\sigma_{1}^{2}(\Pi_{1})$ is $(+1,-3,+2)$, which is identical with the subsequence $(*2)$ .

 Thus, we can apply the results in the Proposition \ref{prop:cutcalc} for the actions of $\underline{W'}$ and $\overline{V'}$ for the subsequence $(*)$.
 
Therefore, the arc  
\[ \underline{A_{m}}\sigma_{1}^{-a_{m}} (\sigma_{1}^{a_{m}}\underline{A_{m-1}} \,\underline{A_{m-2}}\cdots \underline{A_{0}}A_{-1}) (\Gamma_{1}) \]
has the initial segment
\[
 \underline{A_{m}} \sigma_{1}^{-a_{m}}(\underbrace {+1,+1,\ldots, +1}_{(m \slash 2)-1},\underbrace{+1,-2,+3}_{(*)} ) = (\underbrace {+1,+1,\cdots}_{(m \slash 2)-1}, \underline{A_{m}}( \Pi_{1} ) ) 
\]
 if $a_{m}=1$ and
\[ 
 \underline{A_{m}} \sigma_{1}^{-a_{m}}(\underbrace {+1,+1,\ldots, +1}_{(m \slash 2)-1},\underbrace{+1,-3,+2}_{(*)} )  = (\underbrace {+1,+1,\ldots}_{(m \slash 2)-1}, \underline{A_{m}}( \Pi_{1} ) ) 
\]
if $a_{m}=2$.

Thus in either case, the initial segment of the tight cutting sequence of the arc 
\[ \underline{A_{m}}\sigma_{1}^{-a_{m}} (\sigma_{1}^{a_{m}}\underline{A_{m-1}} \,\underline{A_{m-2}}\cdots \underline{A_{0}}A_{-1}) (\Gamma_{1}) \]
 is given by 
 \[ (\underbrace {+1,+1,\ldots,+1}_{(m \slash 2)-1}, \underline{A_{m}}( \Pi_{1} ) )\] 
where $\underline{A_{m}}(\Pi_{1})$ is an initial segment of the tight cutting sequence presentation of arc $\underline{A_{m}}(\Pi_{1})$. 

Similarly, the tight cutting sequence of the arc  
\[ \overline{C_{m}} \sigma_{1}^{-1} (\sigma_{1}d_{[m-1]}^{p} d_{[m-2]}^{p} \cdots d_{[0]}^{p} A_{-1})(\Gamma_{1})
\]
has the initial segment 
\[ (\underbrace {+1,+1,\ldots.+1}_{(m \slash 2)-1}, \overline{C_{m}}( \Pi_{1} ) ) \] 
where $\overline{C_{m}}(\Pi_{1})$ is a initial segment of the tight cutting sequence presentation of arc $\overline{C_{m}}(\Pi_{1})$.

By inductive hypothesis, the arc $\underline{A_{m}}(\Pi_{1})$ left side of the arc $\overline{C_{m}}(\Pi_{1})$, so we conclude that
\[\alpha < \overline{V} < \overline{C_{m}}d_{[m-1]}^{p}\cdots < \underline{A_{m}}\,\underline{A_{m-1}}\cdots < \underline{W}< \beta . \] 

\end{proof}

 Next we proceed to a computation of the order-type of the well-ordered set $(B_{n}^{+},<)$ for arbitrary finite Thurston type orderings.
First we observe the following lemma.
 
\begin{lem}
\label{lem:conjugate}
Let $<$, $<'$ be conjugate left-invariant total orderings of $B_{n}$, having the Property $S$. Then the order type of $(B_{n}^{+},<)$ and $(B_{n}^{+},<')$ are the same.
\end{lem} 

\begin{proof}  
First we remark that the assumption that the ordering has the property $S$ implies its restriction to $B_{n}^{+}$ defines a well-ordering.
Let $\alpha$ be a conjugating element between $<$ and $<'$. From Lemma \ref{lem:conjpos}, we can choose $\alpha$ as a positive braid. From Property $S$, the map $\Phi:(B_{n}^{+},<') \rightarrow (B_{n}^{+},<)$ defined by $\Phi(\beta)=\beta\alpha$ is an order-preserving injection, so the order-type of $(B_{n}^{+}, <')$ is smaller than that of $(B_{n}^{+}, <)$. By interchanging the role of $<$ and $<'$, we also obtain that the order type of $(B_{n}^{+}, <)$ is smaller than that of $(B_{n}^{+}, <')$. Therefore they have the same order-type.
\end{proof} 

Thus, it is sufficient to compute the order-type of $(B_{n}^{+},<)$ for a normal finite Thurston type ordering $<$ to determine the order-type of $(B_{n}^{+},<')$ for general normal finite Thurston type ordering $<'$.

\begin{proof}[Proof of Theorem \ref{thm:ordertype}]

We prove the theorem by induction on $n$. The case $n=2$ is trivial. Let $<$ be a normal finite Thurston type ordering on $B_{n}$ and let $\textsf{Ncode}(n,<)$ be the set of all codes of the $\C$-normal forms with respect to the ordering $<$. 

Theorem \ref{thm:main} asserts that $(\textsf{Ncode}(n,<),\lo)$ is order-isomorphic to $(B_{n}^{+},<)$. A direct computation shows that the order-type of $(\Code(n,<),\lo)$ is 
$\omega^{\omega^{n-2}}$. Since $(\textsf{Ncode}(n,<),\lo)$ is a subset of $(\Code(n,<),\lo)$, the order type of $(\textsf{Ncode}(n,<),\lo )$ is at most $\omega^{\omega^{n-2}}$.

To show the converse inequality, we consider the set of even codes  
\[
\textsf{Code}_{\textsf{even}}(n,<) = 
\left\{
\C \in \Code(n,<) 
\left|
\begin{array}{l} 
\C \textrm{ contains only }\\
\textrm{positive even integers  } \\
\textrm{when neglecting trivial zeros} 
\end{array}
\right.
\right\}.
\]
  That is, a code $\C$ belongs to the set $\textsf{Code}_{\textsf{even}}(n,<)$ if and only if each entry of $\C$ is strictly positive even integer except trivial zeros.
  By inductive hypothesis, a direct computation shows that the order type of $(\textsf{Code}_{\textsf{even}},\lo)$ is $\omega^{\omega^{n-2}}$. 
On the other hand, Tits conjecture, proved in \cite{cp}, asserts the subgroup of $B_{n}$ generated by the square of generators $\{\sigma_{1}^{2},\sigma_{2}^{2},\ldots,\sigma_{n-1}^{2}\}$ has the only trivial commutative relations $\sigma_{i}^{2}\sigma_{j}^{2} = \sigma_{j}^{2}\sigma_{i}^{2} \;(|i-j| >1)$. This implies the set of even codes $\textsf{Code}_{\textsf{even}}$ is a subset of $\textsf{NCode}(n,<)$.
So the order type of $(\textsf{Ncode}(n,<),\lo)$ is at least $\omega^{\omega^{n-2}}$.
Therefore the order type of $(\textsf{Ncode}(n,<),\lo)$ is $\omega^{\omega^{n-2}}$.
\end{proof}

\section{Relationships between the $\C$-normal forms and Dehornoy's $\Phi$-normal form}

   In this section we study the relationships between the $\C$-normal form and Dehornoy's $\Phi$-normal form (alternate normal form) or Burckel's normal form. This gives an algebraic description, the efficient method of computations, and the computational complexities of the $\C$-normal forms.
 
\subsection{$\Phi$-normal form of braid groups}

   The $\Phi$-normal form is a normal form of positive braids defined by Dehornoy in \cite{d2}.
   First we review the definition of the $\Phi$-normal form of the positive braid monoid $B_{n}^{+}$. The construction of such a normal form is valid for much wider class of monoids, called a locally Garside monoid. Since we are interested in the braid groups, we only describes the positive braid monoid case.
   
As is well known, the positive braid monoid $B_{n}^{+}$ defines a Garside structure of $B_{n}$. This implies the positive braid monoid $B_{n}^{+}$ has many good properties. Here we do not describe what a Garside structure is. Only we need is existence and uniqueness of the maximal right divisor.

For a subset $I$ of $\{1,2,\ldots, n-1\}$, let $B_{I}^{+}$ be the submonoid of $B_{n}^{+}$ generated by $\{\sigma_{i}\: | \: i \in I\}$. For each $\beta \in B_{n}^{+}$, there exists the unique maximal right divisor of $\beta$ which belongs to $B_{I}^{+}$. That is, there exists the unique element $\beta' \in B_{I}$ which satisfies the following two properties:
\begin{enumerate}
\item $ \beta \beta'^{-1}\in B_{n}^{+}$.
\item For $\beta'' \in B_{I}^{+}$, if $\beta \beta''^{-1} \in B_{n}^{+}$ then $\beta'\beta''^{-1} \in B_{I}^{+}$.
\end{enumerate} 

  We denote the maximal right divisor of $\beta$ which belongs to $B_{I}^{+}$ by $\beta \wedge B_{I}^{+}$. 

Let $I= \{2,3,\ldots,n-1\}$ and $J = \{1,2,\ldots,n-2\}$.
For a positive $n$-braid $\beta$, the $(B_{I}^{+},B_{J}^{+})$-decomposition (the alternate decomposition) of $\beta$ is a factorization of $\beta$ given by 
\[ \beta = \beta_{m}\beta_{m-1}\cdots \beta_{0}, \]
where $\beta_{i}$ is defined by the inductive formula
\[
\left\{
\begin{array}{lll}
\beta_{0} & = & \beta \wedge B_{I}^{+} \\
\beta_{i} & = & (\beta\cdot \beta_{0}^{-1}\beta_{1}^{-1}\cdots \beta_{i-1}^{-1}) \wedge B_{J}^{+} \;\;\; i:odd \\
         & = & (\beta\cdot \beta_{0}^{-1}\beta_{1}^{-1}\cdots \beta_{i-1}^{-1}) \wedge B_{I}^{+} \;\;\; i:even.
\end{array}
\right.
\]

 Since the word length of $\beta$ is finite, the above factorization must stop in finite step.
 
 For even $i>0$, we identify $B_{I}^{+}$ with $B_{n-1}^{+}$ by the flip map $D$. Similarly, we identify $B_{J}$ with $B_{n-1}^{+}$ by the identity map for odd $i>0$. For $i=0$, we use another identification. We identify $B_{I}^{+}$ with $B_{n-1}^{+}$ by the shift map $Sh$. Then by using these identifications, we can perform the above alternate decompositions for each $\beta_{i}$.
By iterating the alternate decomposition until each $B_{I}^{+}$ is generated by only one generator, we finally obtain the unique positive word representative of $\beta$. We call this unique positive word representative the $\Phi$-{\it normal form} (alternate normal form).

\begin{exam}
Let $\beta = \sigma_{1}\sigma_{3}\sigma_{2}\sigma_{3}\sigma_{2}^{2}\sigma_{1}^{2}\sigma_{3}$ be a positive 4-braid.
The alternate decomposition of $\beta$ is given by
 \[ \beta = \beta_{2}\beta_{1}\beta_{0}=(\sigma_{3}^{3})(\sigma_{2}^{2}\sigma_{1})(\sigma_{2}\sigma_{3}^{2}). \]
 
 Now we iterate the alternate decomposition for each $\beta_{i}$.
 Using an identification $D(\beta_{2})=D(\sigma_{3}^{3}) = \sigma_{1}^{3} \in B_{3}^{+}$, the alternate decomposition of $\beta_{2}$ is given by $\beta_{2}=(\sigma_{3}^{3})$.
 Similarly, the alternate decomposition of $\beta_{1}=\sigma_{2}^{2}\sigma_{1}$ is given by $\beta_{1}= (\sigma_{2}^{2})(\sigma_{1})$. Finally, by using an identification $Sh(\beta_{0})=Sh(\sigma_{2}\sigma_{3}^{2})=\sigma_{1}\sigma_{2}^{2}$, the alternate decomposition of $\beta_{0}$ is given by $\beta_{0} = (\sigma_{2})(\sigma_{3}^{2})$.
 Summarizing, we obtain the $\Phi$-normal form of $\beta$,
 \[ \beta = ((\sigma_{3}^{3}))((\sigma_{2}^{2})(\sigma_{1}))((\sigma_{2})(\sigma_{3}^{2})) = \sigma_{3}^{3}\sigma_{2}^{2}\sigma_{1}\sigma_{2}\sigma_{3}^{2}. \]
 
Notice that this $\Phi$-normal form coincide with the $\C$-normal form given in Example \ref{exam:dehornoy4}. 
\end{exam}

  As this example suggests the $\C$-normal form and Dehornoy's $\Phi$-normal form coincide. 
  
\begin{prop}
\label{prop:m-normal}
   Let $<_{D}$ be the Dehornoy ordering of the braid group $B_{n}$.
Then for each element $\beta$ in $B_{n}^{+}$, the $\C$-normal form $W(\beta;<_{D})$ of $\beta$ with respect to the ordering $<_{D}$ coincide with Dehornoy's $\Phi$-normal form. 
\end{prop}

\begin{proof}
Let $\beta$ in $B_{n}^{+}$ and $W=A_{m}A_{m-1}\cdots A_{-1}$ be the $\C$-normal form of $\beta$. Let us denote the alternate decomposition of $\beta$ by $\beta=\beta_{k}\beta_{k-1}\cdots\beta_{0}$.
 From the definition of the $\C$-normal form, $A_{0}$ is always an empty word. Since $A_{-1} \in \langle \sigma_{2},\sigma_{3},\ldots,\sigma_{n-1} \rangle$, we obtain $\beta_{0}=FA_{-1}$ for some $F \in \langle \sigma_{2},\sigma_{3},\ldots,\sigma_{n-1} \rangle$. If $F$ is not a trivial braid, then the subword decomposition $W= \cdots | FA_{-1}$ defines a strictly bigger code, so it leads a contradiction. So we conclude
\[ A_{-1}=\beta_{0} =  W \wedge \langle \sigma_{2},\sigma_{3},\ldots,\sigma_{n-1} \rangle. \] 
By the similar argument, we obtain $\beta_{i} = A_{i}$ for $i>0$. An iteration of this argument for each $A_{i}$ shows that the $\C$-normal form and the $\Phi$-normal form coincide.
\end{proof}

\begin{rem}
  The code defined in this paper corresponds to {\it the associated exponent sequence} in Dehornoy's paper \cite{d2}.
\end{rem}

Since Dehornoy shows Dehornoy's $\Phi$-normal form coincide with Burckel's normal form \cite{d2}, we obtain an alternative description of Burckel's normal form.

\begin{cor}
 The $\C$-normal form with respect to the Dehornoy ordering is identical with Burckel's normal form. 
\end{cor}

\subsection{Tail-twisted $\Phi$-normal form and $\C$-normal form}

   Now we generalize Proposition \ref{prop:m-normal} for general normal finite Thurston type orderings.
To describe the $\C$-normal form for general normal finite Thurston type orderings in view of Dehornor's $\Phi$-normal form, we introduce a generalized notion of the $\Phi$-normal form, the {\it Tail-twisted $\Phi$-normal form}.

Let $\bk=\{k(1),k(2),\ldots,k(n-1)\}$ be a permutation of $n-1$ integers $\{1,2,\ldots,n-1\}$.
For $\beta \in B_{n}^{+}$, the $\bk$-tail twisted $\Phi$-normal form is a positive word representative of $\beta$ defined by the following inductive procedure. 

As in the definition of codes, we begin with the case $n=2$. 
In this case, we simply define the $\bk$-tail normal form of a positive braid $\sigma_{1}^{p}$ by the word $\sigma_{1}^{p}$.

Assume that we have already constructed the $\bk$-tail twisted normal forms for all $B_{i}^{+}$ and for all permutations in $S_{i}$ $(i<n)$. Here $S_{i}$ is the permutation group of degree $i$. 
Then we define the $\bk$-tail twisted normal form of positive $n$-braids $\beta$ for $\bk \in S_{n}$ as follows.  

First we define the permutations $\bk_{i}$ for $i=0,1,\cdots,n-2$. 
Let $\bk_{0}$ the permutation of $n-2$ integers defined by
\[ \bk_{0} = \{k(j_{1})-1,k(j_{2})-1,\cdots, k(j_{n-2})-1 \: |\: j_{i}< j_{i+1},\; k(j_{i}) > 1\}. \]

Let $m_{j}$ and $M_{j}$ be integers and $I_{j}$ be the set of integers appeared in the definition of the code (See section 3.1 again).
Let $\bk_{j}$ be the permutation of $(M_{j}-m_{j})$ integers, defined by 
\[ \bk_{j} =   \{ k(i_{1}),k(i_{2}),\ldots, k(i_{M_{j}-m_{j}}) \: |\: j<i_{q}< i_{q+1},\; m_{j} \leq k(i_{q}) \leq M_{j}\}. \]

Next we consider a twisted version of the alternate decomposition.
For $\beta \in B_{n}^{+}$, let us define
\[
\left\{
\begin{array}{l}
\beta_{\textsf{Tail}} = \beta \wedge \langle \sigma_{1},\sigma_{2},\ldots,\sigma_{k(1)-1},\sigma_{k(1)+1},\ldots,\sigma_{n-1} \rangle \\
\beta_{T_{0}}= (\beta \beta_{\textsf{Tail}}^{-1}) \wedge \langle \sigma_{2},\sigma_{3},\ldots,\sigma_{n-1} \rangle \\
\beta_{\textsf{Main}}= \beta \beta_{T_{0}}^{-1}\beta_{\textsf{Tail}}^{-1}
\end{array}
\right.
\]
We decompose $\beta_{\textsf{Tail}}$ by $\beta_{\textsf{Tail}} = \beta_{T_{1}}\beta_{T_{2}}\cdots\beta_{T_{n-2}}$ as the following way.
\[
\left\{
\begin{array}{l}
\beta_{T_{n-2}} =  \beta_{\textsf{Tail}} \wedge B_{I_{n-2}}^{+} \\
\beta_{T_{n-j}}  =  (\beta_{\textsf{Tail}}\cdot \beta_{T_{n-2}}^{-1}\beta_{T_{n-3}}^{-1}\cdots \beta_{T_{n-j+1}}^{-1}) \wedge B_{I_{n-j}}^{+}
\end{array}
\right.
\]

We denote the $\Phi$-normal form of the braid $\beta_{\textsf{Main}}$ by $N(\beta_{\textsf{Main}})$.
By some powers of the shift map, each $B_{I_{j}}^{+}$ is naturally identified with $B_{M_{j}-m_{j}+1}^{+}$. Using this identification, we denote by $\widetilde{N}(\beta_{T_{j}};\bk_{j})$ the $\bk_{j}$-tail twisted $\Phi$-normal form of $\beta_{T_{j}}$.
 By the inductive assumption, we have already defined these tail-twisted $\Phi$-normal forms. 
 Now, the {\it $\bk$-tail twisted $\Phi$-normal form} of a positive $n$-braid $\beta$ is a positive word representative $\widetilde{N}(\beta;\bk)$ defined by 
\[
\widetilde{N} (\beta;\bk) = N(\beta_{\textsf{Main}})\widetilde{N}(\beta_{T_{0}};\bk_{0})\cdots\widetilde{N}(\beta_{T_{n-2}};\bk_{n-2}).
\]

 We remark that for the trivial permutation $\bk=\{1,2,\ldots,n-1\}$, the $\bk$-tail twisted $\Phi$-normal form is nothing but the usual Dehornoy's $\Phi$-normal form.
  Using the tail-twisted $\Phi$-normal form, Proposition \ref{prop:m-normal} is generalized as follows.

\begin{prop}
Let $<$ be a normal finite Thurston type ordering of $B_{n}$ which corresponds to the permutation $\bk = \{k(1),k(2),\ldots,k(n-1)\}$.
Then for each $\beta \in B_{n}^{+}$, the $\bk$-tail twisted $\Phi$-normal form $\widetilde{N}(\beta;\bk)$ coincide with the $\C$-normal form $W(\beta;<)$ .
\end{prop}

\begin{proof}
The proof is almost the same as the proof of Proposition \ref{prop:m-normal}, so we only describe the correspondence between these two normal forms. 
For a positive $n$-braid $\beta$, let $W=A_{m}\cdots A_{-1}$ be the $\C$-normal form of $\beta$, and denote the subword decomposition of $A_{-1}$ by $A_{-1}=X_{1}X_{2}\cdots X_{n-2}$.
Similarly, let $\beta=\beta_{\textsf{Main}}\beta_{T_{0}}\beta_{T_{1}}\cdots\beta_{T_{n-2}}$ be a twisted version of the alternate decomposition of $\beta$, and $\beta_{\textsf{Main}}=\beta_{p}\beta_{p-1}\cdots \beta_{0}$ be the alternate decomposition of $\beta_{\textsf{Main}}$. As is easily observed, the braid $\beta_{0}$ is always trivial, so we omit it.
By the same arguments of the proof of Proposition \ref{prop:m-normal}, we obtain a correspondence $\beta_{\textsf{Tail}}=A_{-1}$, $\beta_{T_{0}} = A_{0}$, $\beta_{i}=A_{i}$, and $\beta_{T_{i}}=X_{i}$.
Thus, an inductive argument shows these two normal forms are indeed identical.
\end{proof}

\begin{exam}
The definition of the tail-twisted $\Phi$-normal form is a bit complex, 
so we give a concrete example of the tail-twisted normal form in a simple case.
Let $\bk=\{2,1,3\} \in S_{4}$. The normal finite Thurston type ordering corresponding to the permutation $\bk$ is the ordering given in Example \ref{exam:thurston4}.

Consider a positive $4$-braid $\beta=\sigma_{1}\sigma_{3}\sigma_{2}\sigma_{3}\sigma_{2}^{2}\sigma_{1}^{2}\sigma_{3}$ which appeared in Example \ref{exam:thurston4}. By computation, 
\[ \beta_{\textsf{Tail}}= \beta \wedge \langle \sigma_{1}, \sigma_{3} \rangle = \sigma_{1}^{2}\sigma_{3}^{2}. \]
So $\beta_{T_{1}} = \sigma_{1}^{2}$ and $\beta_{T_{2}}= \sigma_{3}^{2}$ and
\[ \beta_{T_{0}} = \beta\beta_{\textsf{Tail}}^{-1} \wedge \langle \sigma_{2},\sigma_{3} \rangle = \sigma_{3}^{3}\sigma_{2}. \]
Thus,
\[ \beta_{\textsf{Main}} = \beta \beta_{\textsf{Tail}}^{-1} \beta_{T_{0}}^{-1} = \sigma_{1}. \]

Since for this permutation $\bk$, all permutations $\bk_{0}$, $\bk_{1}$ and $\bk_{2}$ are the trivial permutations, so all of the tail-twisted $\Phi$-normal forms with respect to these permutations are merely the ordinary $\Phi$-normal forms.
Thus, we obtain
\[
\left\{
\begin{array}{l}
\widetilde{N}(\beta_{T_{1}};\bk_{1}) = \sigma_{1}^{2}, \; \widetilde{N}(\beta_{T_{2}};\bk_{2}) = \sigma_{3}^{2}\\
\widetilde{N}(\beta_{T_{0}};\bk_{0}) = \sigma_{3}^{3}\sigma_{2} \\
N(\beta_{\textsf{Main}}) = \sigma_{1}.
\end{array}
\right.
\]

Combing the above results, we conclude that the $\bk$-tail twisted normal form of $\beta$ is given by
\[ \widetilde{N}(\beta ; \bk) = \sigma_{1}\sigma_{3}^{3}\sigma_{2}\sigma_{1}^{2}\sigma_{3}^{2}. \]
which is of course identical with the $\C$-normal form of $\beta$ given in Example \ref{exam:thurston4}.
\end{exam}

Finally, using these results and Dehornoy's results about the $\Phi$-normal forms, we prove Theorem \ref{thm:complexity}. 
Theorem \ref{thm:complexity} is interesting in the following sense. Although in \cite{rw} a computational complexity of finite Thurston-type orderings is given, but it relies on Mosher's automatic structure of mapping class groups \cite{mo}. So the relationships between a usual word representative and a computational complexity is indirect. The tail-twisted $\Phi$-normal form formulation allows us to determine the computational complexity in usual word representatives, and provide more convenient and practical method to compute finite Thurston type orderings.
  
\begin{proof}[Proof of Theorem \ref{thm:complexity}]
   For Dehornoy's $\Phi$-normal form, these results are proved in \cite{d2}. First we consider normal case. In this case, the $\C$-normal form is identical with the tail-twisted $\Phi$-normal form. A computation of twisted tail parts, that is, a computation of $\beta_{T_{i}}$ and tail-twisted normal forms $\widetilde{N}(\beta_{T_{i}};\bk_{i})$, is merely a computation of maximal right divisors and the tail-twisted $\Phi$-normal forms for smaller $n$. So the computation of these parts does not affect the order of the total computational complexity. Therefore the theorem also holds for normal Thurston type orderings. 
For a general finite Thurston type ordering, computation of $\C$-normal form for word length $l$ braid is a computation of $\C$-normal form with respect to normal finite type of a braid with the word length $l+p$, where $p$ is the word length of a positive conjugating element. Since $p$ is a constant, therefore the conclusion holds.  
\end{proof}

   The $\C$-normal form constructions and arguments in this paper can be seen as a geometric background of the $\Phi$-normal form and Burckel's normal form. Our $\C$-normal form brings new information about the questions raised by Dehornoy in \cite{d2}. Our construction gives an answer to the question $6.5$ in \cite{d2}, which is in short, ``Give a direct proof of the fact that the $\Phi$-normal form gives a combinatorial description of the Dehornoy ordering."
Here the word ``direct" means without using Burckel's result of the Dehornoy orderings. Our proof requires neither Burckel's normal form nor its tree construction. We use a Geometric definition of the Dehornoy ordering given in \cite{fgrrw}, \cite{sw} and the property $S$, which is confirmed by the geometric definition. Thus our proof is independent of Burckel's results.
Of course, the above question is still open when we restrict to use the usual algebraic definition of the Dehornoy ordering and algebraic method.

\end{document}